\documentclass[]{article}
\usepackage[utf8]{inputenc}
\usepackage{amsfonts,amsmath,amsthm}
\usepackage{geometry}
\usepackage{tikz}
\usepackage{todonotes}
\usepackage{accents}
\usepackage{mathtools}
\usepackage{booktabs}
\usepackage{array}
\mathtoolsset{showonlyrefs}
\usepackage{graphicx}
\usepackage{epic}
\usepackage[Option]{overpic}
\usepackage[font={small,it}]{caption}
\usepackage{subcaption}
\usepackage{float}
\usepackage{physics}
\usepackage{pifont}
\usepackage{dsfont}

\numberwithin{equation}{section}

\newtheorem{theorem}{Theorem}[section]

\newtheorem{lemma}[theorem]{Lemma}
\newtheorem{remark}[theorem]{Remark}
\newtheorem{proposition}[theorem]{Proposition}

\newtheorem{definition}[theorem]{Definition}

\newcommand{\rmd}{\mathrm{d}}
\newcommand{\rmD}{\mathrm{D}}
\renewcommand{\epsilon}{\varepsilon}
\newcommand{\Id}{\mathrm{Id}}
\newcommand{\ddt}[1]{\frac{\textnormal{d}#1}{\textnormal{d}t}}

\newcommand{\R}{\mathbb{R}}
\newcommand{\N}{\mathbb{N}}
\newcommand{\ve}{\varepsilon}
\newcommand{\cO}{\mathcal{O}}
\newcommand{\txta}{\textnormal{a}}
\newcommand{\txtr}{\textnormal{r}}
\newcommand{\cS}{\mathcal{S}}
\newcommand{\cD}{\mathcal{D}}
\newcommand{\cY}{\mathcal{Y}}

\newcommand{\nwc}{\newcommand}
\nwc{\red}[1]{\textcolor{black}{#1}}
\nwc{\blue}[1]{{\color{black}{#1}}}

\makeatletter
\renewcommand*\env@matrix[1][\arraystretch]{%
  \edef\arraystretch{#1}%
  \hskip -\arraycolsep
  \let\@ifnextchar\new@ifnextchar
  \array{*\c@MaxMatrixCols c}}
\makeatother

\providecommand{\keywords}[1]
{
  \textbf{{Keywords: }} #1
}

\providecommand{\mathclassification}[1]
{
  \textbf{{Mathematics  Subject  Classification  (2010): }} #1
}

\begin{document}
\title{
Extended and symmetric loss of stability for canards in planar fast-slow maps
}
\author{Maximilian Engel\thanks{
Technische Universität München,
Forschungseinheit Dynamics,
Zentrum Mathematik, M8, 
Boltzmannstraße 3,
85748 Garching bei München. 
(\texttt{maximilian.engel@tum.de})
}
 \ and  
Hildeberto Jard\'on-Kojakhmetov\thanks{
Faculty of Science and Engineering, Bernoulli Institute, University of Groningen, 9747 AG, Groningen, The Netherlands 
(\texttt{h.jardon.kojakhmetov@rug.nl})
}
}

\pagenumbering{arabic}

\maketitle

\begin{abstract}
    We study fast-slow maps obtained by discretization of planar fast-slow systems in continuous time. We focus on describing the so-called delayed loss of stability induced by the slow passage through a singularity in fast-slow systems. This delayed loss of stability can be related to the presence of canard solutions. Here we consider three types of singularities: transcritical, pitchfork, and fold. First, we show that under an explicit Runge-Kutta discretization the delay in loss of stability, due to slow passage through a transcritical or a pitchfork singularity, can be arbitrarily long. In contrast, we prove that under a Kahan-Hirota-Kimura discretization scheme, the delayed loss of stability related to all three singularities is completely symmetric in the linearized approximation, in perfect accordance with the continuous-time setting. 
    
\end{abstract}
\noindent\keywords{Delayed loss of stability, discretization, maps, canards, fast-slow systems.}

\noindent\mathclassification{34E15, 34E17, 37G10, 39A12.}



\section{Introduction}
Consider a system of singularly perturbed ordinary differential equations (ODEs) in \textit{fast} time scale
\begin{equation}\label{fastequ}
    \begin{split}
    \ddt{x} = x' &= f(x,y,\epsilon)\,, \\
    \ddt{y} = y' &= \epsilon g(x,y,\epsilon)\,,  \quad \ x \in \mathbb{R}^m, 
\quad y \in \mathbb R^n, \quad 0 < \epsilon \ll 1\,,
    \end{split}
\end{equation}
with \textit{critical manifold} 
\begin{equation}
    \cS_0= \{(x,y) \in \mathbb{R}^{m+n} \,:\, f(x,y,0) = 0 \}\,.
\end{equation}
The set $\cS_0$ is called \textit{normally hyperbolic} if 
the matrix $\textnormal{D}_xf(p)\in\mathbb{R}^{m\times m}$ for all $p\in \cS_0$ has no
spectrum on the imaginary axis. For a normally hyperbolic and compact $\cS_0$, \textit{Fenichel 
Theory} \cite{Fenichel4,Jones,ku2015,WigginsIM} implies that for $\epsilon$ 
sufficiently small, there is a locally invariant \emph{slow manifold} $\cS_{\epsilon}$ 
behaving like a regular perturbation of $\cS_0$. On the other hand, \emph{loss of normal hyperbolicity}, which occurs whenever $\textnormal{D}_xf(p)$ has at least one eigenvalue on the imaginary axis, is known  to be responsible for many complicated dynamic effects, such as relaxation oscillations and mixed mode oscillations \cite{desroches2012mixed,ku2015}, that are difficult to analyze rigorously. In this paper we will focus on one of such features found when normal hyperbolicity is lost, namely \emph{canards}. In particular, we shall study planar fast-slow systems with a \emph{canard point} at the origin, past whom trajectories connect an attracting branch of the slow manifold with a repelling one, also described as \emph{maximal canard}~\cite{BenoitCallotDienerDiener,DuRo96,ks2001/3}. For continuous-time fast-slow systems, these canard solutions allow us to explain why one observes a delay in the onset of instabilities when trajectories slowly cross a singularity \cite{de2008maximum,de2016entry,hayes2016geometric}, see also \cite{fruchard2009survey,lobry1991dynamic,neishtadt1987persistence,neishtadt1988persistence,neishtadt2009stability}. 

In this paper we focus on (discretized) planar fast-slow systems with canard points associated to three singularities: the transcritical, the pitchfork, and the fold. One important motivation to consider fast-slow systems with the aforementioned singularities is that they are common in models used in the applied sciences, thus organising the behaviour of the corresponding dynamics and playing a crucial role for numerical simulations. For example, transcritical singularities may appear in epidemiological and other population dynamic models \cite{boudjellaba2009dynamic} and seem to be a generic mechanism for stability loss in some network models \cite{jardon2019fast}; pitchfork bifurcations can be found e.g. in decision making dynamics \cite{gray2018multiagent} or biochemical oscillators \cite{tsaneva2006diffusion}, while the fold singularity is frequently encountered in neuron models \cite{bertram2017multi,de2015neural,mitry2013excitable}.

Extending and further quantifying observations in~\cite{ArciEngelKuehn19} and \cite{EngelKuehn18} (where only the forward-Euler scheme is taken into account), we consider various, more general, time discretizations of equation~\eqref{fastequ} and investigate the linearized behaviour along their corresponding canard solutions. \blue{Note that the linearizations of these maps along trajetories give non-autonomous discrete time dynamical systems whose properties are strongly linked to the linear stability behavior of the corresponding methods but not a priori explained by them: discretization and linearization do not necessarily commute and, hence, the different methods deserve a detailed analysis in this context.}

For example, the canonical form of the transcritical singularity in a fast-slow system \blue{for which a maximal canard exists} reads (up to leading order)
\begin{align} \label{ODE_transcrit_overview}
\begin{array}{r@{\;\,=\;\,}l}
x' & x^2 - y^2 + \blue{\lambda} \epsilon, \\
y' & \epsilon,
\end{array}
\end{align}
\blue{where $\lambda$ is varied around $1$, depending on higher order terms and $\ve$. Taking $\lambda=1$ for the simplest case,}
its forward-Euler discretization induces the map $P:\R^2\to\R^2$ given by
\begin{equation}
	P(x,y) = \left(x+h\left(x^2-y^2\right)+h\ve,y+h\ve\right).
\end{equation}
Analogously to the time-continuous case, the set $\left\{ x^2=y^2 \right\}$ is invariant under the iteration of $P$. In particular, \blue{for $x_0 < 0$}, the orbit
\begin{equation}
	\gamma_{x_0}(n)=(x_n,y_n)=(x_0+nh\ve,x_0+nh\ve),\quad n\in\N,
\end{equation}
corresponds to the ``discrete-time maximal canard'' in \cite{EngelKuehn18}, \blue{starting on the attracting branch $\{x = y < 0 \}$ and continuing on the repelling branch $\{x = y > 0 \}$.}
As we will prove in Section \ref{sec:rk_transcrit}, one now observes that trajectories that start close to the attracting part of such a canard, stay close to the repelling part of the canard for very long times; something one cannot observe in the time-continuous case. 
\blue{The effect of higher order terms in this generic canonical form can be neglected locally such that our calculations and results stay valid for the general case. In particular, when using the blow-up method for dealing with the non-hyperbolic singularity at the origin, in the rescaling chart at the origin analysis of equation~\eqref{ODE_transcrit_overview} is all that is needed to understand the dynamics around the singularity (see e.g.~\cite{ks2001/2}).}

In the case of the pitchfork canonical form \blue{(up to leading order) and for which a maximal canard exists}
\begin{align*} 
\begin{array}{r@{\;\,=\;\,}l}
x' &= x(y - x^2) \blue{+ \lambda \ve},  \\
y' &= \epsilon \,,
\end{array}
\end{align*}
\blue{the parameter $\lambda$ has the same role as in the transcritical case, now varying around $0$. Taking $\lambda=0$,} the set $\left\{ x=0 \right\}$ is invariant, and the forward-Euler discretization map~\eqref{Euler_pitchfork} has the canard trajectory
\begin{equation*}
    \gamma_{y_0}(n)=(0,y_0+nh\ve),\qquad n\in\N.
\end{equation*} 
We observe the same extended delay effect along the \blue{maximal} canard as in the transcritical case (cf.~\cite{ArciEngelKuehn19}), as we will make precise in Section~\ref{sec:pitchfork}.

Generally, for explicit Runge-Kutta methods, the same effects appear. These delay times of the onset of instability can in fact grow arbitrarily large as we will prove in Section~\ref{sec:rk_transcrit} by investigating contraction and expansion rates along the linearization of the maximal canard. The main result concerning arbitrarily long delay of bifurcations for transcritical and pitchfork singularities can be formulated in the following Theorem, which is sketched in Figure \ref{fig:discrete1}.

\begin{theorem}[Delay for explicit RK schemes] \label{thm:RungeKutta}

Consider an explicit Runge-Kutta discretization, with step size $h>0$, of equation~\eqref{fastequ} being the canonical form of a fast-slow system with transcritical or pitchfork singularity, with parameters such that there is a canard solution.
%
Denoting by $x_0=- \rho <0$ the initial $x$-coordinate, we have the following for the transcritical canard (and analogously for the $y$-coordinate in the pitchfork case):
\begin{enumerate}

\item For any $h,\ve>0$, there exists a maximal canard trajectory $\gamma_{-\rho}$ for the discrete-time system induced by the Runge-Kutta scheme.

\item If $x^*$ denotes the $x$-coordinate where the contraction/expansion rate around $\gamma_{-\rho}$ is compensated, i.e.~where trajectories, starting close enough to $\gamma_{-\rho}$, leave a vicinity of $\gamma_{-\rho}$, then there exist values $(\rho^*,h^*,\ve^*)$ such that
\begin{equation}
\lim_{(\rho,h,\ve) \to (\rho^*,h^*,\ve^*)} x^* \to \infty.
\end{equation}

\item There is a particular lower bound  $K^*$ for the number of iterations corresponding with $x^*$, expressed by the Lambert $W$ function.
\end{enumerate}
\end{theorem}
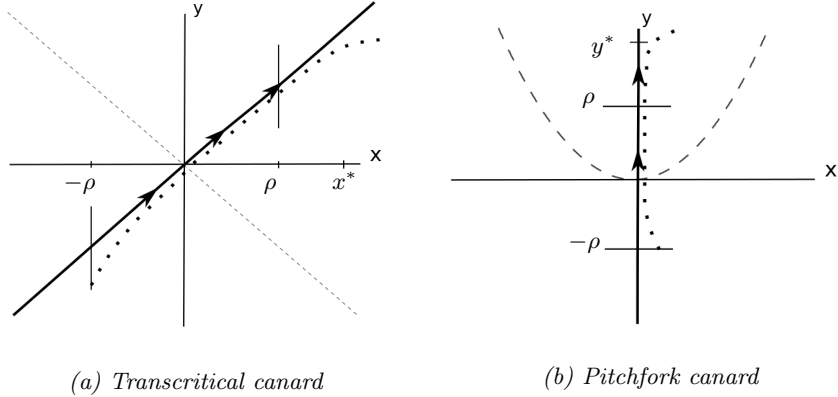
\begin{figure}[htbp]
        \centering
        \begin{subfigure}{.4\textwidth}
        \centering
  		\begin{overpic}[width=1.0\textwidth]{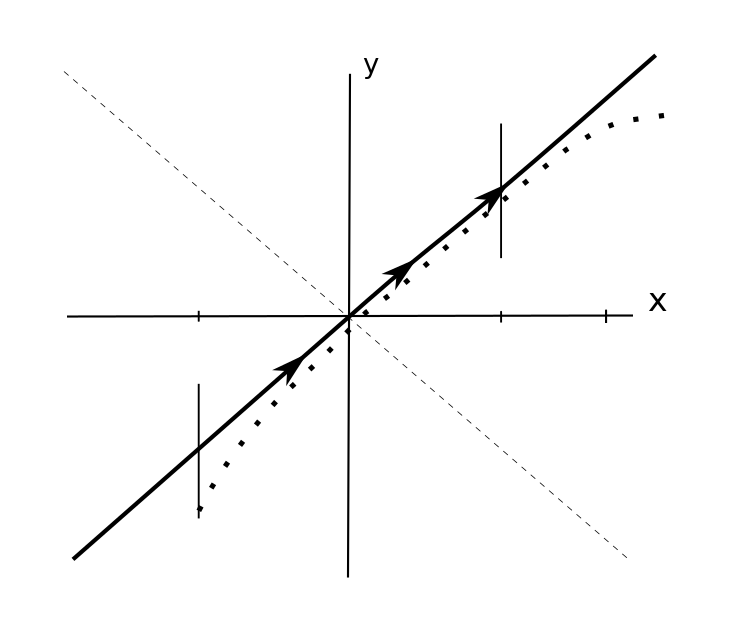}            
                \put(21,36){\small $- \rho$}
                \put(65,36){\small $ \rho$}
                \put(80,36){\small $x^*$}
        \end{overpic}
        \caption{Transcritical canard}
		\end{subfigure}%
        \begin{subfigure}{.4\textwidth}
        \centering
  		\begin{overpic}[width=1.0\textwidth]{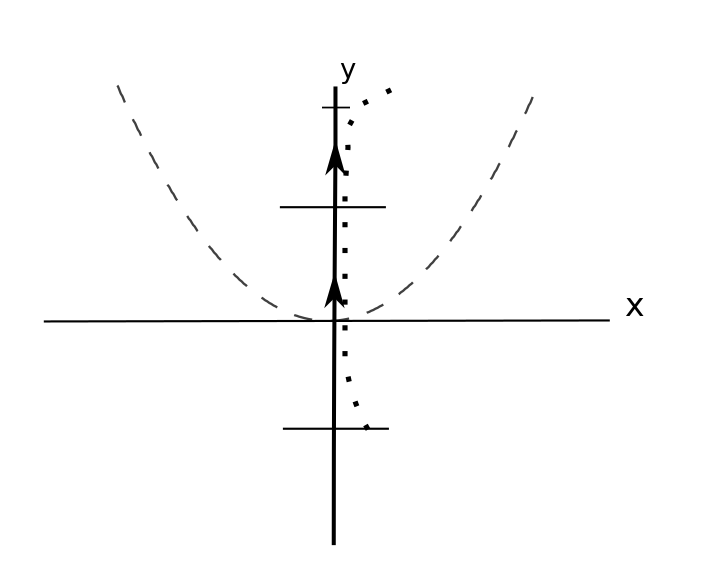}   
  				\put(32,22){\small $- \rho$}
                \put(35,54){\small $ \rho$}
                \put(37,65){\small $y^*$}
        \end{overpic}
        \caption{Pitchfork canard}
		\end{subfigure}
		\caption{The sketches illustrate Theorem~\ref{thm:RungeKutta}. The solid line with arrows indicates a canard, while the dotted curve is a nearby trajectory. The point $x^*$ (resp. $y^*$ for the pitchfork) indicates the $x$-coordinate where the contraction and the expansion along the canard have compensated each other.
		We show that, when an explicit RK-discretization is employed, the delayed loss of stability in planar fast-slow maps with transcritical and pitchfork singularities is not symmetric (in contrast to their continuous-time analogues). In fact the delay can be arbitrarily large. See more details in Section \ref{sec:RK} and a numerical example in Section \ref{sec:numerics}. }
        \label{fig:discrete1}
\end{figure}
For ODEs of the form~\eqref{fastequ}, the delay of bifurcation (for example for Hopf bifurcations \cite{hayes2016geometric,neishtadt1987persistence,neishtadt1988persistence}) can be characterised by the \emph{way-in/way-out map} (or \emph{input-output function}) derived in the following way (cf.~\cite[Section 12.2]{ku2015}): define the phase 
$$ \Psi(\tau):= \int_{\tau_*}^{\tau} \lambda_1(y^0(s) ) \rmd s, $$
where $y^0(s)$ denotes the slow-flow solution for equation~\eqref{fastequ}, \blue{$\lambda_1(y)$ the first eigenvalue of $(\rmD_x f) (y)$ at the critical set $\cS_0$} and $\tau_*$ is such that $y^0(\tau_*) =0$, i.e.~where $y^0(s)$ is passing through the bifurcation point. Then for $\tau$ sufficiently close to $\tau_*$, one can define the \emph{way-in/way-out map} $\Pi$ as the function that maps a time $\tau < \tau_*$ to a time $\Pi(\tau) > \tau_*$ by the condition
\begin{equation}\label{conttime:wayinwayout}
    \Re [\Psi(\tau)] = \Re [\Psi(\Pi(\tau))].
\end{equation}
The essential statement of Theorem~\ref{thm:RungeKutta} is that for explicit Runge-Kutta discretizations of transcritical and pitchfork canards, the way-in/way-out map, which gives precisely $\Pi(\tau)= \tau$ in the ODE situation, will explode for certain parameter regimes of $\rho$ and $h$.

In the second part of the paper, we will see that the bilinear Kahan-Hirota-Kimura discretization scheme, also abbreviated as \emph{Kahan method} and known for preserving conserved quantities in various cases (see e.g.~\cite{PetreraPfadlerSuris09, PetreraPfadlerSuris11, PetreraSuris18}), is more suitable for finding a way-in/way-out map with analagous properties to the ODE situation.
We show in Section~\ref{sec:kahan_transcritical} that the Kahan method, which is in fact an implicit Runge-Kutta scheme that yields an explicit form for quadratic vector fields, preserves precisely the time-continuous structure of the transcritical canard in the sense, that the time trajectories spend near the attracting part of the maximal canard is the same that they spend near the repelling part. In other words, the discrete-time way-in/way-out map, defined in analogy to~\eqref{conttime:wayinwayout}, is completely symmetric in this case.

The Kahan method is not explicit in the case of the pitchfork singularity due to the cubic term which seems to make the analysis more cumbersome than for the transcritical case.
However, we can still analyze the linearization along the maximal canard in this, now implicit, discretization and obtain analogous results to the transcritical case. \blue{For the pitchfork problem, we will additionally embed the Kahan method into a more general scheme of A-stable, symmetric second-order methods, demonstrating that a part of these methods has the desired preservation properties and the other part does not, depending on an additional parameter.}
In the case of a folded canard, again the Kahan method turns out to be an accurate behaviour-preserving discretization scheme for the canard problem \cite{Engeletal2019}. In this case, the explicit RK-methods do not even give a discrete-time singular canard trajectory such that the discrete-time dynamics could not be analyzed as before. Hence, for the folded canard, we only focus on the Kahan map (see Section~\ref{sec:fold}).

We can summarise the results on the symmetry of the way-in/way-out map for the Kahan discretization of transcritical, pitchfork and fold singularities in the following Theorem, which is sketched in Figure \ref{fig:discrete2}.
\begin{theorem}[Symmetry for Kahan method] \label{thm:Kahan}
For any $h, \epsilon > 0$ (except for one singular combination), consider the entry point $x(0)= - \rho < 0$ (and $y(0)= - \rho < 0$ in the pitchfork case). Then we have:

\begin{enumerate}
    \item There exists a maximal canard trajectory $\gamma_{-\rho}$ for the discrete-time system induced by the Kahan scheme.
    
    \item If, in the case of the transcritical or pitchfork singularity, we have  $\rho = \epsilon h N + \epsilon h/2$ for some $N \in \mathbb{N}$, or in the case of the fold singularity we have $\rho = \epsilon h N/2$ for some $N \in \mathbb{N}$, then
the way-in/way-out map $\psi_h$ for $\gamma_{-\rho}$
is well defined and takes the value 
$$\psi_h(-N) = N.$$
Otherwise, 
\begin{equation*}
\psi_h(-N) \in \{ N+1, N+2\}.
\end{equation*}
In other words, the expansion has compensated for contraction at 
$$x^* \in \{\rho - \epsilon h, \rho, \rho + \epsilon h\}, \quad ( y^* \in \{\rho - \epsilon h, \rho, \rho + \epsilon h\} \ \text{in the pitchfork case} )$$
giving full symmetry of the entry-exit relation.
    
\end{enumerate}

\end{theorem}
\begin{figure}[htbp]
        \centering
        \begin{subfigure}{.4\textwidth}
        \centering
  		\begin{overpic}[width=1.0\textwidth]{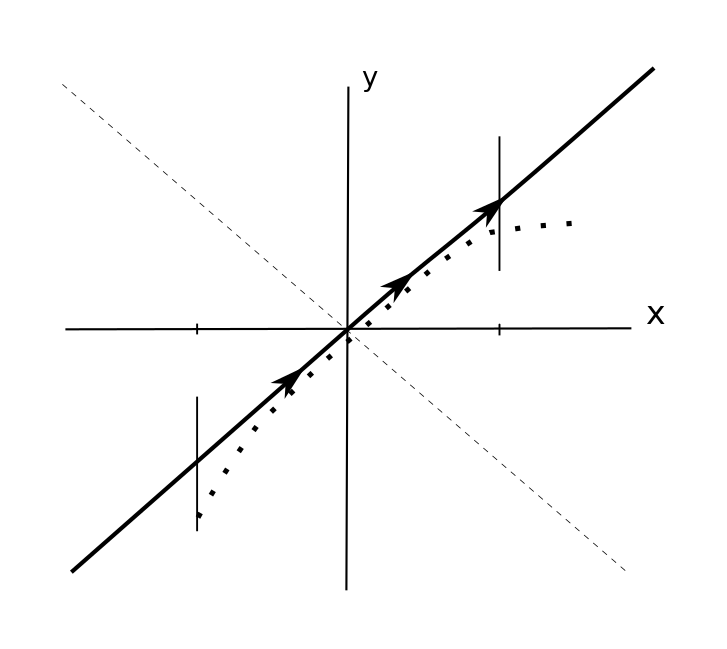}            
                \put(21,38){\small $- \rho$}
                \put(67,38){\small $ \rho$}
        \end{overpic}
        \caption{Transcritical canard}
		\end{subfigure}%
        \begin{subfigure}{.4\textwidth}
        \centering
  		\begin{overpic}[width=1.0\textwidth]{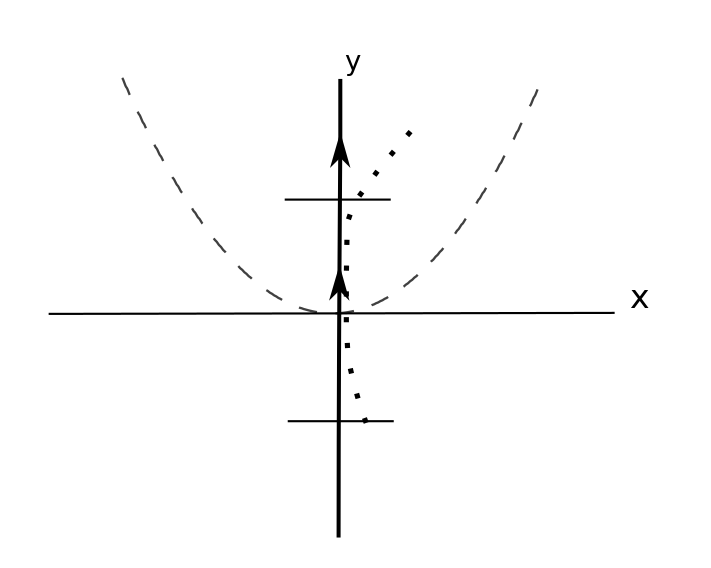}   
  				\put(30,20){\small $- \rho$}
                \put(35,52){\small $ \rho$}
        \end{overpic}
        \caption{Pitchfork canard}
		\end{subfigure}
		\begin{subfigure}{.4\textwidth}
        \centering
  		\begin{overpic}[width=1.0\textwidth]{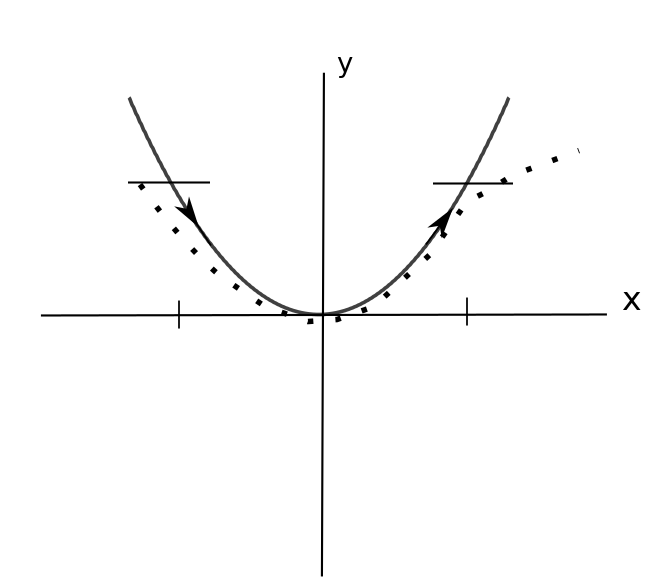}            
                \put(21,35){\small $- \rho$}
                \put(67,35){\small $ \rho$}
        \end{overpic}
        \caption{Folded canard}
        \label{discrete_lambdasmall}
		\end{subfigure}%
		\caption{The sketches illustrate Theorem~\ref{thm:Kahan}. The solid curve with arrows indicates a canard, while the dotted curve is a nearby trajectory. We show that, when a Kahan discretization scheme is employed, the delayed loss of stability in planar fast-slow maps with transcritical, pitchfork, and fold singularities is symmetric (in accordance with their continuous-time analogues). See more details in Section \ref{sec:Kahan} and a numerical example in Section \ref{sec:numerics}.}
        \label{fig:discrete2}
\end{figure}
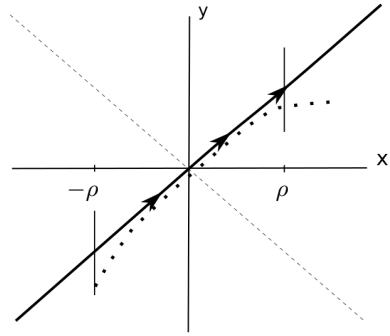
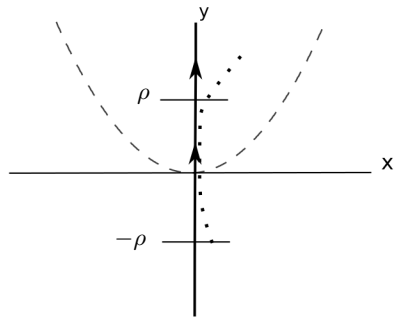
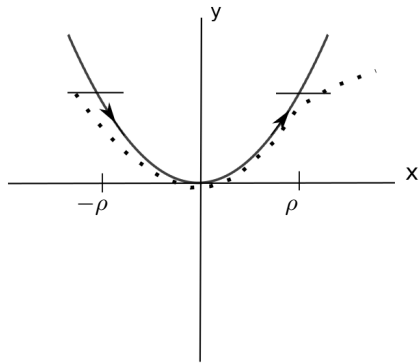

In the forthcoming Sections we prove Theorems \ref{thm:RungeKutta} and \ref{thm:Kahan}. More specifically, Section \ref{sec:RK} is dedicated to the analysis of fast-slow planar maps under explicit Runge-Kutta discretization while in Section \ref{sec:Kahan} we consider the Kahan-Hirota-Kimura scheme. Afterwards, in Section \ref{sec:numerics} we exemplify our theoretical results with a numerical simulation. We conclude this paper in Section \ref{sec:discussion} with a brief discussion.

\section{Explicit Runge-Kutta methods and extended loss of stability} \label{sec:RK}

This section is dedicated to the study of\textbf{} delayed loss of stability under explicit Runge-Kutta (RK) discretization schemes. In order to clearly present the main ideas, we first consider a fast-slow transcritical singularity under a forward Euler discretization, where we show the possibility of arbitrarily delayed loss of stability. Next, we show that the same conclusion holds for general explicit RK methods. Later we present a similar result for the pitchfork singularity.

\subsection{Transcritical singularity} \label{sec:rk_transcrit}

The canonical form  of the transcritical singularity in a fast-slow system reads as
\begin{align*}
\begin{array}{r@{\;\,=\;\,}l}
x' & x^2 - y^2 + \blue{\lambda} \ve \blue{+h_1(x,y, \epsilon)}, \\
y' & \epsilon \blue{( 1 + h_2(x,y, \epsilon))}\,,
\end{array}
\end{align*}
where $0 < \epsilon \ll 1$ is the time scale separation parameter, \blue{$\lambda \approx 1$ is an unfolding parameter for canards and
\begin{align*}
    h_1(x,y,\epsilon) &= \mathcal O \left(x^3, x^2y, xy^2, y^3, \epsilon x, \epsilon y, \epsilon^2 \right) \,,\\
     h_2(x,y,\epsilon) &= \mathcal O \left(x, y, \epsilon \right).
\end{align*}
We will neglect the terms $h_1$ and $h_2$ in the following since, locally around the origin, they can be understood as small perturbations not changing the dynamical behavior; we can therefore also set $\lambda=1$. Hence, our model system, being easily generalized to any transcritical fast-slow problem, reads
\begin{align} \label{ODE_transcrit}
\begin{array}{r@{\;\,=\;\,}l}
x' & x^2 - y^2 + \epsilon, \\
y' & \epsilon\,.
\end{array}
\end{align}
}

\subsubsection{Euler method}

In this section we consider the iterated map obtained after forward-Euler discretization of the canonical form of a fast-slow transcritical singularity, see \cite{EngelKuehn18} for the details. In particular, we consider the so-called canard case. To be precise, let us consider the map $P:\R^2\to\R^2$ given by
\begin{equation}
	P(x,y) = \left(x+h\left(x^2-y^2\right)+h\ve,y+h\ve\right),
\end{equation}
where $0<\ve\ll1$, and $h>0$. We are interested in the dynamical system defined by iterations of the map $P$, that is $P^n(x_0,y_0)$ where $(x_0,y_0)\in\R^2$ are initial conditions and $n\in\N$. Note that the set $\cS=\left\{ (x,y)\in\R^2\, : \, x^2=y^2 \right\}$ is invariant under the iteration of $P$. In particular, \blue{for $x_0 < 0$}, the trajectory
\begin{equation}
	\gamma_{x_0}(n)=(x_n,y_n)=(x_0+nh\ve,x_0+nh\ve),\quad n\in\N,
\end{equation}
corresponds to the ``discrete-time maximal canard'' in \cite{EngelKuehn18}, \blue{starting on the attracting branch $\{x = y < 0 \}$ and continuing on the repelling branch $\{x = y > 0 \}$}, and shall play an essential role in our analysis.

\begin{remark}
	Our goal is to give details on the contraction/expansion rate around $\gamma_{x_0}$ in terms of $(x_0,\ve,h)$. This is motivated by \cite{EngelKuehn18}, where besides a thorough analysis of the fast-slow discrete time transcritical singularity, it is shown that the contraction towards the maximal canard is stronger in discrete time compared to its counterpart in continuous time, see \cite[Theorem 3.1, (T3)]{EngelKuehn18}.
\end{remark}

Proceeding with our analysis, we note that $\gamma_{x_0}|_{\left\{ \ve=0\right\}}$ is attracting for $x_0<0$ and repelling for $x_0>0$. Therefore, let $\rho\in O(1)$ be a positive constant and set $x_0=-\rho$. Next, we compute the variational equation of $P^n$ along $\gamma_{x_0}$, which yields
\begin{equation}\label{eq:variational}
	v(n+1)= \begin{pmatrix}
		1+2(-\rho+nh\ve)h & -2(-\rho+nh\ve)h\\ 0 & 1
		\end{pmatrix}v(n),
\end{equation}
where $v(n)=(v_1(n),\,v_2(n))\in\R^2$ and $n\in\N$. The local contraction/expansion rate is characterised by the solutions of \eqref{eq:variational} in the transversal (hyperbolic) direction, corresponding to the eigenvector $(1,0)^\top$. Thus, the solution of \eqref{eq:variational} with initial condition $(v_1(0),\,v_2(0))=(1,0)$ is given by $(v_1(n),0)$ where
\begin{equation}\label{eq:trans-rate}
v_1(n)=\prod_{k=0}^{n-1}\left(1-2h(\rho-kh\ve)\right), \quad n\in\N,\, n>0.	
\end{equation}

It is precisely the number given by $v_1(n)$ that provides the overall contraction/expansion rate near $\gamma_{x_0}$. As an example, suppose that there exists a $K\in\N$ that is the solution of
\begin{equation}
1=\prod_{k=0}^{K-1}\left(1-2h(\rho-kh\ve)\right).	
\end{equation}

Then $K$ gives the number of iterations after which the contraction towards $\gamma_{x_0}$ and the expansion away from $\gamma_{x_0}$, in the $x$-direction, have compensated each other. In other words, the time $K$ is the discrete analogue of the \emph{asymptotic moment of jumping} \cite{neishtadt2009stability} in continuous time.  We can prove the following.

\begin{proposition}\label{prop:trans-euler} Consider the inequality
	\begin{equation}\label{eq:ineq1}
		\prod_{k=0}^{K-1}\left(1-2h(\rho-kh\ve)\right)\geq1,	
	\end{equation}
	where $K\in\N$ and $K>1$. If $1-2h\rho>0$ and the inequality holds, then
	\begin{equation}\label{eq:KstarEuler}
		K\geq K^*=\frac{1}{h^2\ve}\left( -1+2h\rho + \exp\left(W\left(-h^2\ve\ln(1-2\rho h)\right)\right)\right), 
	\end{equation}
	where $W$ denotes the Lambert W function. Furthermore, let $x^*=-\rho+K^*h\ve$ for some fixed values of $(h,\ve)$. Notice that $x^*$ is the $x$-coordinate where the contraction/expansion rate around the canard is compensated. Then \begin{equation}
	\lim_{\rho\to\frac{1}{2h}} x^* = \infty.
    \end{equation}

\end{proposition}

\begin{proof} To simplify the notations let $a:=1-2h\rho>0$, $b:=2h^2\ve>0$. Then \eqref{eq:ineq1} reads as
\begin{equation}\label{eq:ineq2}
	\prod_{k=0}^{K-1} (a+bk)\geq 1.
\end{equation}
Taking the natural logarithm on both sides of \eqref{eq:ineq2}, and isolating the term for $k=0$, one gets
\begin{equation}
	\begin{split}
 		\sum_{k=1}^{K-1}\ln(a+bk) &\geq -\ln(a).
 	\end{split}
 \end{equation} 
Next, using the finite form of Jensen's inequality\footnote{Let $\phi$ be a real, concave function and $\left\{ x_i \right\}$ a sequence of numbers in the domain of $\phi$. Then the finite form of Jensen's inequality that we use is $ \phi\left( \frac{\sum x_i}{n} \right)\geq \sum\frac{\phi(x_i)}{n}$. } one has
\begin{equation}\label{eq:ineq3}
	\ln\left( \frac{1}{K-1}\sum_{k=1}^{K-1}a+bk \right)\geq \frac{1}{K-1}\sum_{k=1}^{K-1}\ln(a+bk) \geq -\frac{1}{K-1}\ln(a).
\end{equation}
Hence, by disregarding the middle term in equation~\eqref{eq:ineq3}, and by simplifying the arithmetic series $\sum_{k=1}^{K-1}a+bk$, one obtains 
\begin{equation}
 	\begin{split}
 		K\ln\left( a+\frac{b}{2}K \right)>(K-1)\ln\left( a+\frac{b}{2}K   \right) &\geq -\ln(a).
	\end{split}
 \end{equation} 
Next, we disregard the middle term of the last inequality and define $z\in\R$ by $\exp(z)=a+\frac{b}{2}K$, leading to
\begin{equation}
	\begin{split}
		\exp(z)z>(\exp(z)-a)z>-\frac{b}{2}\ln(a),
	\end{split}
\end{equation}
where the last inequality holds due to $a\in(0,1)$. Thus, lower bound $z^*$ for $z$ is given by the solution of $\exp(z^*)z^*=-\frac{b}{2}\ln(a)$. Using the Lambert W function~\cite[Eq.~4.13]{NIST:DLMF} one gets $z^*=W\left( -\frac{b}{2}\ln(a)\right)$, which in turn provides a lower bound $K^*$, that is $K\geq K^*$, where
\begin{equation}
	K^*=\frac{2}{b}\left( \exp\left( W\left( -\frac{b}{2}\ln(a)\right)\right) - a \right).
\end{equation}
The proof is completed by re-substituting the values of $a$ and $b$. The computation of the limit follows from the fact that $\lim_{x\to\infty}W(x)=\infty$.%
\end{proof}
\blue{\begin{remark} \label{rem:euler}
Note that the condition $2\rho h < 1$ is in exact accordance with the stability criterion for the Euler method with respect to the Dahlquist test equation. The delay of exit from the repelling part of the critical curve goes to infinity when the boundary of this stability region is approached since the stabilization factors at the attracting parts of the critical set become larger and larger. In this way the effect of the asymmetric structure of the linearized non-autonomous dynamics along the two opposite sides of the origin is increased more and more. Hence, the crucial reason for the stabilization is the asymmetric nature of the linearization which then also affects the behavior of the full nonlinear problem.
\end{remark}}
\blue{\begin{remark} \label{rem:euler_hot}
Furthemore, we would like to point out that for small $h$ extreme delays are only expected for relatively large $\rho$. This might seem to contradict our argument for only using model system~\eqref{ODE_transcrit}, neglecting potential higher order terms due to their local insignificance. However, in a blow-up analysis at the origin, higher order terms can also be neglected for very large $\rho$ such that this case becomes generally relevant.
\end{remark}}
Next we extend the previous analysis to general explicit Runge-Kutta discretization schemes. 

\subsubsection{General explicit Runge-Kutta schemes}\label{sec:trans-RK}

Given a planar system
\begin{equation}
   \begin{split}
       x' &= f(x,y),\\
       y' &= g(x,y),
   \end{split}
\end{equation}
the $s$-stage \emph{explicit} Runge-Kutta method is given by
\begin{equation}
   \begin{split}
       x_{n+1} &=  x_n + h\sum_{i=1}^s\alpha_i\kappa_i\\
       y_{n+1} &=  y_n + h\sum_{i=1}^s\alpha_i\ell_i,
   \end{split}
\end{equation}
where
\begin{equation}
   \begin{split}
       \kappa_i &=f\left( x_n+h\sum_{j=1}^{i-1}a_{ij}\kappa_j, y_n+h\sum_{j=1}^{i-1} a_{ij}\ell_j \right)\\
       \ell_i &=g\left( x_n+h\sum_{j=1}^{i-1}a_{ij}\kappa_j, y_n+h\sum_{j=1}^{i-1} a_{ij}\ell_j \right),
   \end{split}
\end{equation}
and the coefficients $\alpha_i\in\R$ and $a_{ij}\in\R$ depend on the RK-scheme and can be chosen from the so-called \emph{Butcher tableau}. Thus, for a fast-slow system with a transcritical singularity \eqref{ODE_transcrit} one has
\begin{equation}\label{eq:rk-trans}
   \begin{split}
       x_{n+1} &= x_n + h\sum_{i=1}^s\alpha_i\kappa_i\\
       y_{n+1} &= y_n + h\sum_{i=1}^s\alpha_i\ell_i,
   \end{split}
\end{equation}
where $\ell_i=\ve$ and
\begin{equation}
   \kappa_i=\left( x_n+h\sum_{j=1}^{i-1}a_{ij}\kappa_j \right)^2 - \left(y_n+h\ve\sum_{j=1}^{i-1} a_{ij}\right)^2+\ve.
\end{equation}
\begin{remark}
The forward-Euler discretization is the $1$-stage explicit RK method.
\end{remark}

First we show that, similar to the Euler discretization, the maximal canard  $\gamma_{x_0}(n)=(x_n,y_n)=(x_n,x_n)$ is a solution of the $s$-stage RK discretization of the fast-slow transcritical singularity. To be more precise,  from \eqref{eq:rk-trans} let us consider the map $P_s:\R^2\to\R^2$ given by 
\begin{equation}\label{eq:P-trans}
 P_s(x,y)=\left( x + h\sum_{i=1}^s\alpha_i\kappa_i, \, y + h\ve\sum_{i=1}^s\alpha_i \right), \quad 
\end{equation}%
where
\begin{equation}\label{eq:P-k}
    \kappa_i=\left( x+h\sum_{j=1}^{i-1}a_{ij}\kappa_j \right)^2 - \left(y+h\ve\sum_{j=1}^{i-1} a
    _{ij}\right)^2+\ve.
\end{equation}
It is clear that iterations of the map $P_s$ define a discrete-time dynamical system. Next we show that, just as in the forward Euler case, the subset \blue{$\cD =\left\{ (x,y)\in\R^2\,:\, x=y \right\}\subset\cS$} is invariant under the iterations of $P_s$ for any $s\geq1$.
\begin{proposition}\label{prop:RK-canard} Consider \eqref{eq:P-trans}. Then $\cD =\left\{ (x,y)\in\R^2\,|\, x=y \right\}$ is invariant under the iterations of $P_s$. Moreover, if $\sum_{i=1}^s\alpha_i=1$, then the solutions along $\cD$ are given by $\gamma_{x(0)}(n)=(x(0)+\ve h n, x(0)+\ve h n)$.
\end{proposition}
\begin{proof}
   Let $x=y$, then from \eqref{eq:P-k} it follows that
   \begin{equation}\label{eq:proof-ki}
       \kappa_i = 2xh\sum_{j=1}^{i-1}a_{ij}\kappa_j-2xh\ve\sum_{j=1}^{i-1}a_{ij}+h^2\left( \sum_{j=1}^{i-1}a_{ij}k_j\right)^2-h^2\ve^2\left( \sum_{j=1}^{i-1}a_{ij}\right)^2+\ve.
   \end{equation}
    Note that $\kappa_1=\kappa_2=\ve$. Next, assume that $\kappa_1=\cdots=\kappa_{i-1}=\ve$. Then, it follows immediately from equation \eqref{eq:proof-ki} that $\kappa_i=\ve$, for all $i=1,\ldots,s$. This implies
   \begin{equation}
       P_s(x,x) =\left( x+h\ve\sum_{i=1}^s\alpha_i,\, x+h\ve\sum_{i=1}^s\alpha_i\right),
   \end{equation}
   which shows that indeed the diagonal $\cD$ is invariant. The expression of $\gamma_{x(0)}(n)$ follows from the common assumption that $\sum_{i=1}^s\alpha_i=1$ which implies that $y_{n+1} = y_n + h\sum_{i=1}^s\alpha_i\ell_i=y_n + h\ve$, and whose solutions are as indicated. 
\end{proof}
Next, we are going to study the variational equation of \eqref{eq:rk-trans} along $\gamma_{x(0)}(n)$. For shortness of notation let $\hat\gamma=\gamma_{x(0)}(n)$. We have the following.

\begin{lemma}\label{lemm:RK-variational} The variational equation of \eqref{eq:rk-trans} along $\hat\gamma$ is given by
\begin{equation}\label{eq:RK-variational}
    v(n+1)=\begin{pmatrix}
         1+h\sum_{i=1}^s\alpha_i\pdv{\kappa_i}{x_n}|_{\hat\gamma} & h\sum_{i=1}^s\alpha_i\pdv{\kappa_i}{y_n}|_{\hat\gamma} \\
        0 & 1
    \end{pmatrix}v(n).
\end{equation}
Moreover the eigenvector corresponding to the eigenvalue $1$ is $(1,1)^\top$.

\end{lemma}
\begin{proof}
The Jacobian of \eqref{eq:rk-trans} reads as
\begin{equation}
    J=\begin{pmatrix}
    1+h\sum_{i=1}^s\alpha_i\pdv{\kappa_i}{x_n} & h\sum_{i=1}^s\alpha_i\pdv{\kappa_i}{y_n} \\
    0 & 1
    \end{pmatrix}.
\end{equation}
Recall that $k_i|_{\hat\gamma}=\ve$, and let $A_i:=\sum_{j=1}^{i-1}a_{ij}$. Note that $A_1=0$. It follows  that
    \begin{align}
        \pdv{\kappa_i}{x_n}|_{\hat \gamma} &= 2\left( x_n+h\ve A_i\right)\left( 1+h \sum_{j=1}^{i-1}a_{ij}\pdv{\kappa_j}{x_n}|_{\hat\gamma}\right),\label{eq:par1}\\
        \pdv{\kappa_i}{y_n}|_{\hat \gamma} &= 2\left( x_n+h\ve A_i\right)\left( -1+h \sum_{j=1}^{i-1}a_{ij}\pdv{\kappa_j}{y_n}|_{\hat\gamma}\right).
    \end{align}
Next, note that $(1,1)^\top$ is the eigenvector for the eigenvalue $1$ of $J|_{\hat\gamma}$ if and only if $\pdv{\kappa_i}{y_n}|_{\hat\gamma}=-\pdv{\kappa_i}{x_n}|_{\hat\gamma}$ for all $i=1,\ldots,s$. For $i=1$ this holds trivially, for the rest of the terms one proves the equality by induction.
\end{proof}

For notation convenience, let $J_1(x_n)$ denote the first component of $J|_{\hat\gamma}$, that is $J_1(x_n)= 1+hQ_s(x_n)$, where, from \eqref{eq:par1}, we have
\begin{equation}
    Q_s(x_n) = \sum_{i=1}^s\alpha_i\pdv{\kappa_i}{x_n}|_{\hat\gamma} = \sum_{i=1}^s2\alpha_i\left( x_n+h\ve A_i\right)\left( 1+h \sum_{j=1}^{i-1}a_{ij}\pdv{\kappa_j}{x_n}|_{\hat\gamma}\right)
\end{equation}

\begin{remark}\label{rem:Qs}
$Q_s(x_n)$ is, generically, a polynomial in $x_n$ of degree $s$. Also, for sake of simplifying the notation, we are omitting the dependence of $Q_s$ on $(h,\ve)$.
\end{remark}

Recall also that for initial condition $x(0)=-\rho$, we have $\gamma_{-\rho}(n) = (x_n,y_n)=(-\rho + nh\ve,-\rho + nh\ve)$. Just as in the Euler method $J_1(-\rho)$ shall play an important role. In fact, let us define the following.
\begin{definition}\label{def:crit-triplet}
Let $(\rho,h,\ve)=(\rho^*,h^*,\ve^*)$ be a solution of $J_1(-\rho)=0$. Then we call $(\rho^*,h^*,\ve^*)$ \emph{a critical triplet}.
\end{definition}

\begin{remark}
For the forward Euler method the critical triplet is $(\rho^*,h^*,\ve^*)=(\frac{1}{2h},h,\ve)$.
\end{remark}

It follows from Lemma \ref{lemm:RK-variational} that the centre eigenvector is aligned with the invariant space $\cD$ along which the maximal canard $\gamma_{x(0)}$ is located. Therefore, to study the contraction/expansion rate along $\gamma_{x(0)}$ we consider the variational equation \eqref{eq:RK-variational} in the transverse direction to $(1,1)^\top$. For this, it suffices to consider \eqref{eq:RK-variational} with initial condition $v(0)=(1,0)$. Thus, with this setup, the solution $v_1(n)$ of the variational equation \eqref{eq:RK-variational} is given by
\begin{equation}\label{eq:RK-v1}
    v_1(n) = \prod_{k=0}^{n-1}  \left( 1+hQ_s(-\rho+kh\ve) \right).
\end{equation}
\begin{remark}
Note that \eqref{eq:variational} is obtained by choosing $s=1$ and the corresponding constants of the scheme in \eqref{eq:RK-v1}. Indeed for $s=1$, one has $\alpha_1=1$ and thus $Q_1(-\rho)=-2\rho$.
\end{remark}

\begin{proposition}\label{prop:RK-trans}
Consider the inequality
	\begin{equation}\label{eq:ineq1_rk}
		\prod_{k=0}^{K-1}\left( 1+hQ_s(-\rho+kh\ve) \right)\geq1,	
	\end{equation}
	where $K\in\N$ and $K>1$. If $1+hQ_s(-\rho)>0$, $K>2$, and inequality \eqref{eq:ineq1_rk} holds, then
	\begin{equation}\label{eq:RK-Kstar}
	   K\geq  K^* = 1+\exp\left( W\left( -\frac{\log(1+hQ_s(-\rho))}{\bar C (s+1)} \right) \right),
	\end{equation}
	where $W$ denotes the Lambert W function, and $\bar C$ is a constant that is determined by the choice of the RK scheme. Assume that a critical triplet $(\rho^*,h^*,\ve^*)$ exists\footnote{That is, the equation $1+hQ_s(-\rho)=0$ has at least one real triplet solution.}. Then  $\lim_{(\rho,h,\ve)\to(\rho^*,h^*,\ve^*)}K^*=\infty$. This means that if we define $x^*=-\rho+K^*h\ve$ (which is the $x$-coordinate at which the contraction/expansion rate is compensated), we have the limit \begin{equation}
	\lim_{(\rho,h,\ve)\to(\rho^*,h^*,\ve^*)} x^* = \infty.
	\end{equation}
\end{proposition}
\begin{proof}
Recalling Remark \ref{rem:Qs} and for simplicity of notation, let us write \eqref{eq:ineq1_rk} as
\begin{equation}\label{eq:RK-e1}
    \prod_{k=0}^{K-1}\sum_{i=0}^s\theta_i k^i\geq1,
\end{equation}
for some coefficients $\theta_i=\theta_i(h,\ve,R_s)$ where in $R_s$ we gather the coefficients of the RK scheme. It is worth noting that $\theta_0=1+hQ_s(-\rho)$. 
Taking $\log$ on both sides of \eqref{eq:RK-e1} and proceeding similar to the proof of Proposition \ref{prop:trans-euler} we get
\begin{equation}
    \sum_{k=1}^{K-1}\log\left(\sum_{i=0}^{s} \theta_i k^i \right) \geq -\log\theta_0.
\end{equation}

Next, again using the finite form of Jensen's inequality allows us to write
\begin{equation}\label{eq:RK-e2}
   (K-1) \log\left( \frac{1}{K-1}\sum_{k=1}^{K-1}\sum_{i=1}^s \theta_i k^i \right)\geq-\log\theta_0.
\end{equation}

Noting that the sums on the left side of equation~\eqref{eq:RK-e2} do commute, we then have
\begin{equation}\label{eq:RK-e3}
    \begin{split}
        \sum_{k=1}^{K-1}\sum_{i=1}^s \theta_i k^i &= \sum_{i=1}^s\sum_{k=1}^{K-1} \theta_i k^i=\sum_{i=1}^s\theta_i\sum_{k=1}^{K-1}  k^i\\
        &=\sum_{i=1}^s\theta_i\left( \frac{1}{i+1}\sum_{j=0}^i \binom{i+1}{j}B_j (K-1)^{i+1-j}\right),
    \end{split}
\end{equation}
where the last equality is obtained by using Faulhaber's formula and the $B_j$'s are Bernoulli numbers of the second kind \cite[pp. 106-109]{conway2012book}. ubstituting \eqref{eq:RK-e3} in \eqref{eq:RK-e2} leads to
\begin{equation}\label{eq:RK-e4}
\begin{split}
    (K-1)\log\left( \sum_{i=1}^s\theta_i\left( \frac{1}{i+1}\sum_{j=0}^i \binom{i+1}{j}B_j (K-1)^{i-j}\right) \right) & \geq - \log\theta_0,
\end{split}
\end{equation}
where we note that we have cancelled out the term $\frac{1}{K-1}$ of the left hand side of  \eqref{eq:RK-e2} with the appropriate one in \eqref{eq:RK-e3}. For  convenience let us write 
\begin{equation}
    \sum_{i=1}^s\theta_i\left( \frac{1}{i+1}\sum_{j=0}^i \binom{i+1}{j}B_j (K-1)^{i-j}\right)=\sum_{i=1}^s C_i(K-1)^i,
\end{equation}
where $C_i$, $i=1,\ldots,n$, are coefficients depending on $(h,\ve,R_s)$.
Next, note that for fixed $s$ we have 
\begin{equation}
    \sum_{i=1}^s C_i(K-1)^i=\frac{\max\left\{|C_i|\right\}}{\max\left\{|C_i|\right\}}\sum_{i=1}^s C_i(K-1)^i=\max\left\{|C_i|\right\}\sum_{i=1}^s D_i(K-1)^i,
\end{equation}
where $D_i\coloneqq\frac{C_i}{\max\left\{ |C_i| \right\}}\in[-1,1]$. Hence, one has $\max\left\{|C_i|\right\}\sum_{i=1}^s D_i(K-1)^i\leq \max\left\{|C_i|\right\}\sum_{i=1}^s(K-1)^i$ and therefore, simplifying the geometric series $\sum_{i=1}^s(K-1)^i$ one gets
\begin{align} \label{ineq:cKKstar}
\sum_{i=1}^s C_i(K-1)^i&\leq \max\left\{|C_i|\right\}\frac{(K-1)((K-1)^s-1)}{K-2}.
\end{align}

One should keep in mind that the above geometric series is divergent, so one should fix a finite $s$ for the above formula to make sense. In practical terms this is not an issue since we recall that $s$ is the stage of the RK-method, and this is usually a small positive integer. 

Combining \eqref{eq:RK-e4} and \eqref{ineq:cKKstar}, we get
\begin{equation}\label{eq:RK-e5}
    (K-1)\log\left( \max\left\{|C_i|\right\}\underbrace{\frac{(K-1)((K-1)^s-1)}{K-2}}_{=:F(K)}\right)\geq-\log\theta_0.
\end{equation}
Note that $F(K)>2$ due to the assumption $K>2$. Therefore:
\begin{equation}\label{eq:RK-e6}
    \begin{split}
        \log( \max\left\{|C_i|\right\} F(K) ) &=\left(\frac{\log(\max\left\{|C_i|\right\})}{\log(F(K))}+1\right)\log(F(K))\\
        &\leq \underbrace{\left(\frac{\left|\log(\max\left\{|C_i|\right\})\right| }{\log(2)}+1\right)}_{=:\bar C}\log(F(K)).
    \end{split}
\end{equation}

Furthermore, note that
\begin{equation}\label{eq:RK-e7}
    \begin{split}
        \log (F(K)) &= \log\left( \frac{(K-1)((K-1)^s-1)}{K-2} \right)\\
        &=\log\left( (K-1)((K-1)^s-1) \right) - \log(K-2)\\
        &\leq \log\left((K-1)^{s+1}\right) = (s+1)\log(K-1).
    \end{split}
\end{equation}  

So, combining \eqref{eq:RK-e6} and \eqref{eq:RK-e7} allows us to simplify \eqref{eq:RK-e5} to
\begin{equation}
    (K-1)\log(K-1)\geq-\frac{\log\theta_0}{\bar C (s+1)}.
\end{equation}
Next, using the Lambert W function as we did for the proof of Proposition \ref{prop:trans-euler}, we get the estimate
\begin{equation}\label{eq:RK-e8}
    K\geq K^* = 1+\exp\left( W\left( -\frac{\log\theta_0}{\bar C (s+1)} \right) \right).
\end{equation}
Thus \eqref{eq:RK-Kstar} follows from substituting the value of $\theta_0=1+hQ_s(-\rho)$ in \eqref{eq:RK-e8}. Finally, the argument for the limit of $x^*$ follows from $\lim_{x\to\infty}W(x)=\infty$.
\end{proof}

\begin{remark}
Although the proofs of Propositions \ref{prop:trans-euler} and \ref{prop:RK-trans} are similar, some of the intermediate simplifications are different. Hence, we do not recover the estimate \eqref{eq:KstarEuler} from \eqref{eq:RK-Kstar}.
\end{remark}
\blue{\begin{remark} \label{rem:RK}
The lack of symmetry in the linearization of explicit Runge-Kutta maps causes the full stabilization, as explained in Remark~\ref{rem:euler} for the Euler method.
\end{remark}}

\subsection{Pitchfork singularity} \label{sec:pitchfork}
The canonical form of the pitchfork singularity in a fast-slow system reads
\begin{align*}
\begin{array}{r@{\;\,=\;\,}l}
x' &= x(y - x^2) \blue{+ \lambda \epsilon + h_1(x,y, \epsilon)}\,,  \\
y' &= \epsilon \blue{(1 + h_2(x,y, \epsilon))}\,,
\end{array}
\end{align*}
where, again, $ 0 < \epsilon \ll 1$ is the time scale separation parameter \blue{$\lambda \approx 0$ is an unfolding parameter for canards and
\begin{align*}
    h_1(x,y,\epsilon) &= \mathcal O \left(x^2y, xy^2, y^3, \epsilon x, \epsilon y, \epsilon^2 \right) \,,\\
     h_2(x,y,\epsilon) &= \mathcal O \left(x, y, \epsilon \right).
\end{align*}}
\blue{With the same arguments as for the transcritical case, we reduce the system to the model problem
\begin{align} \label{ODE_pitchfork_overview}
\begin{array}{r@{\;\,=\;\,}l}
x' &= x(y - x^2)  \\
y' &= \epsilon\,.
\end{array}
\end{align}}
The analysis of the expansion-contraction rate for the fast-slow pitchfork singularity under forward Euler discretization is similar to the one performed for the transcritical singularity. Therefore, we just sketch the main arguments, see \cite{ArciEngelKuehn19} for the details.

The map obtained after forward Euler discretization of the fast-slow pitchfork singularity is given by
\begin{equation} \label{Euler_pitchfork}
    P(x,y)=(x+h(x(y-x^2)),y+h\ve).
\end{equation}
In this case, the maximal canard is
\begin{equation}
    \gamma_{y(0)}(n)=(0,y(0)+nh\ve),\qquad n\in\N.
\end{equation}
Letting $y(0)=-\rho$, where $\rho\in\cO(1)$ is a positive constant we find that the expansion-contraction rate along $\gamma_{-\rho}(n)$ is given by (compare with \eqref{eq:trans-rate})
\begin{equation}
    v_1(n)=\prod_{k=0}^{n-1}\left(1+h(-\rho+kh\ve)\right).
\end{equation}
 Then, analogously to Proposition \ref{prop:trans-euler}, for the pitchfork case we have the following.
\begin{proposition}\label{prop:pitch-euler} Consider the inequality
	\begin{equation}\label{eq:ineq2_pitch}
		\prod_{k=0}^{K-1}\left(1+h(-\rho+kh\ve)\right)\geq1,	
	\end{equation}
	where $K\in\N$ and $K>1$. If $1-h\rho>0$ and inequality~\eqref{eq:ineq2_pitch} holds, then
	\begin{equation}
		K\geq K^*:=\frac{1}{h^2\ve}\left(-1+h\rho+\exp\left(W\left(-h^2\ve\ln(1-h\rho )\right)\right)\right), 
	\end{equation}
	where $W$ denotes the Lambert W function. Furthermore, let $y^*=-\rho+K^*h\ve$, then
 \begin{equation}
	\lim_{\rho\to\frac{1}{h}} y^* = \infty.
\end{equation}
\end{proposition}
\begin{proof}
The proof follows the exact same reasoning as the proof of Proposition \ref{prop:trans-euler}, so we do not repeat it here.
\end{proof}

\begin{remark}
Again, just as we elaborated for the transcritical singularity, the arbitrarily delayed loss of stability extends to general explicit Runge-Kutta methods such that analogous results to Propositions \ref{prop:RK-canard} and \ref{prop:RK-trans} hold. We omit the details here as the reasoning is very similar to Section \ref{sec:trans-RK}. \blue{In addition note that the condition $\rho h < 1$ is again in exact accordance with the stability criterion for the Euler method with respect to the Dahlquist test equation. The associated considerations are analogous to Remark~\ref{rem:euler}.}
\end{remark}

We conclude these subsections by noting that Theorem \ref{thm:RungeKutta} immediately follows from Propositions \ref{prop:RK-canard}, \ref{prop:RK-trans} and its analogues for the pitchfork case.

\subsection{Fold singularity} \label{sec:fold}
We consider the canonical form of a fast-slow system with fold singularity, \blue{admitting a canard connection,
\begin{align*}
\begin{array}{r@{\;\,=\;\,}l}
x' & - y h_1(x,y, \epsilon) +  x^2 h_2(x,y, \epsilon)  + \epsilon h_3(x,y, \epsilon), \\
y' & \epsilon(x h_4(x,y, \epsilon) - \lambda h_5(x,y, \epsilon) + y h_6(x,y, \epsilon)),
\end{array}
\end{align*}}
where  $0 < \epsilon \ll 1 $ again quantifies the time scale separation \blue{$\lambda \approx 0$ is an unfolding parameter for canards and
\begin{align} \label{higherorder}
\begin{array}{r@{\;\,=\;\,}l}
h_i(x,y, \epsilon) & 1 + \mathcal{O}(x,y, \epsilon)\,, \quad i=1,2,4,5\\
h_i(x,y, \epsilon) & \mathcal{O}(x,y, \epsilon)\,, \quad i=3,6.
\end{array}
\end{align}}
\blue{With the same arguments as before, we reduce the system to the model problem
\begin{align}  \label{ODE_fold}
\begin{array}{r@{\;\,=\;\,}l}
x' & - y +  x^2 \,, \\
y' & \epsilon x\,.
\end{array}
\end{align}}
We apply the explicit forward-Euler discretization with step size $h > 0$ to system~\eqref{ODE_fold} and obtain a 
map 
given by
\begin{equation} \label{map_Euler}
\begin{pmatrix}[1.7]
x \\ y 
\end{pmatrix}
\mapsto \begin{pmatrix}[1.7]
\tilde{x}  \\ \tilde y 
\end{pmatrix}
= \begin{pmatrix}[1.7]
x + h(x^2  - y ) \\
y + \epsilon h x 
\end{pmatrix}.
\end{equation}
Note that by the change of variables $ \epsilon h \to h$ we can write the system in the slow time scale as
\begin{equation} \label{slowEuler}
\epsilon \frac{\tilde x - x}{h} =  x^2 - y\,, \quad \frac{\tilde y - y}{h} =  x\,.
\end{equation}

In analogy to the time-continuous case, the critical manifold $\cS$ is given as
$$\cS = \left\{(x,y)\in\mathbb{R}^2: y =  x^2\right\},$$ 
splitting into two normally hyperbolic branches, the attracting subset $\cS_\txta = \{(x,y)\in \cS: x<0\} $ and the repelling subset $\cS_\txtr = \{(x,y)\in \cS: x>0\}$. 
It follows from~\cite[Theorem 4.1]{HPS77} that for $\epsilon, 
h > 0$ small enough there are corresponding forward invariant slow manifolds 
$ \cS_{\txta, \epsilon,h}$ and $\cS_{\txtr, \epsilon,h}$. The origin, i.e.~the canard point in the ODE case, is again a non-hyperbolic singularity.

However, we make the following observation (which is a simplified version of \cite[Proposition 3.1]{Engeletal2019}):
\begin{proposition} \label{prop:euler_fold}
The equation of the slow subsystem corresponding with~\eqref{slowEuler} reads
$$ \tilde x^2 =  x^2+ x h  \,,$$
which has the two solutions  $$ \tilde x =  \pm \sqrt{x^2 + x h},$$
on the set
$$\left\{(x,y)\in (\mathbb{R}\setminus \left( -h, 0\right) ) \times \mathbb{R} :  y = x^2\right\} \subset \cS .$$ 
Each solution has a fixed point at $x=0$ as opposed to the continuous-time case where the slow flow follows $\dot x = \frac{1}{2}$.
\end{proposition}
\begin{proof}
Setting $\epsilon=0$ in~\eqref{slowEuler}, the statement follows from a straight forward calculation.
\end{proof}

The previous lemma shows that there is no connection between $\cS_{\txtr}$ and $\cS_{\txta}$, which means that there is no singular canard solution on $\cS$ through the origin. Therefore, using the heuristic argument of the continuous-time case \cite[section 3]{krupa2001extending}, one cannot expect the occurrence of canards for the forward-Euler scheme. Compare also with Section \ref{sec:K-fold} and \cite[Proposition 3.2]{Engeletal2019}.


Additionally, when $\epsilon > 0$, we observe that (cf.~\cite[Remark 2.2]{Engeletal2019}) the ODE system~\eqref{ODE_fold} possesses the conserved quantity
\begin{equation} \label{firstintegral_epsilon}
H(x, y) = \frac{1}{2} e^{-2 y/\epsilon} \left( y - x^2 + \frac{\epsilon}{2} \right)\,,
\end{equation}
vanishing on the invariant set
\begin{equation} \label{invariant_epsi}
\cS_{\epsilon} : = \left\{ (x,y) \in \mathbb{R}^2 \, : \, y = x^2 - \frac{\epsilon}{2} \right\},
\end{equation}
which consists precisely of the attracting branch $\cS_{\txta, \epsilon} = \left\{ (x,y) \in S_{\epsilon} \, : \, x < 0 \right\}$ and the repelling branch $\cS_{\txtr, \epsilon} = \left\{ (x,y) \in S_{\epsilon} \, : \, x > 0 \right\}$, such that trajectories on $\cS_{\epsilon}$ go through the origin with speed $\dot x = \epsilon/2$.
Similarly to Proposition~\ref{prop:euler_fold}, it follows from an easy calculation that there is no function $c(\epsilon,h)$ such that $\{ y= x^2 - \frac{\epsilon}{2} + c(\epsilon, h)\}$ is invariant for the dynamics induced by~\eqref{map_Euler}.

For general $s$-stage RK schemes one argues similarly. In such a case, following \cite[Ch. VI]{wanner1996solving}, the corresponding explicit $s$-stage RK discretization of \eqref{ODE_fold} induces a map $(x,y)\mapsto(\tilde x,\tilde y)$ given by
\begin{equation}\label{eq:RK-fold}
	\begin{split}
		\ve \tilde x & = \ve x + h \sum_{i=1}^s\alpha_i (X_i^2 - Y_i),\\
		\tilde y &= y + h\sum_{i=1}^s\alpha_i X_i,\\
		\ve X_i &= \ve x + h\sum_{j=1}^{i-1}a_{ij}(X_j^2-Y_j),\\
		Y_i &= y + h\sum_{j=1}^{i-1}a_{ij}X_j,
	\end{split}
\end{equation}
\blue{where the coefficients $\alpha_i$ and $a_{ij}$ are given by a particular RK-method. As we did before, we let $A_i=\sum_{j=1}^{i-1}a_{ij}$ and note that $A_1=0$.}

\begin{lemma} Consider \eqref{eq:RK-fold}. 
	In the limit $\ve\to0$, one has $X_i^2-Y_i=0$ for all $i=1,\ldots,s$.
\end{lemma}
\begin{proof}
	First we note that from the third and fourth equations of \eqref{eq:RK-fold} one has $(X_1,Y_1)=(x,y)$. Next, taking the limit $\ve\to0$ in \eqref{eq:RK-fold} leads to
	\begin{equation}
		\begin{split}
		0 & = \sum_{i=1}^s\alpha_i (X_i^2 - Y_i),\\
		\tilde y &= y + h\sum_{i=1}^s\alpha_i X_i,\\
		0 &= \sum_{j=1}^{i-1}a_{ij}(X_j^2-Y_j), \qquad \forall i\in[1,s],\\
		Y_i &= y + h\sum_{j=1}^{i-1}a_{ij}X_j, \qquad \forall i\in[1,s].
		\end{split}
	\end{equation}
 The equations $0 = \sum_{j=1}^{i-1}a_{ij}(X_j^2-Y_j), \; \forall i\in[1,s]$ and $0 = \sum_{i=1}^s\alpha_i (X_i^2 - Y_i)$ lead to the result.
\end{proof}

It follows from the previous lemma that the critical manifold is given by 
\begin{equation}
	\cS = \left\{ X_i^2 = Y_i, \quad i=1,\ldots,s \right\},
\end{equation}
where the solutions can be found iteratively with $(X_1,Y_1)=(x,y)$, $X_i^2=Y_i$ and $Y_i = y + h\sum_{j=1}^{i-1}a_{ij}X_j $. Consequently, the reduced map on the critical manifold is
\begin{equation}\label{eq:RK-fold-red1}
	\tilde x^2 = x^2 + h\sum_{i=1}^s\alpha_i X_i,
\end{equation}
where
\begin{equation}\label{eq:RK-fold-red2}
	X_i = \pm \sqrt{Y_i}, \qquad Y_i = x^2 + h\sum_{j=1}^{i-1}a_{ij}X_j.
\end{equation}

\begin{lemma} Let $s>1$. A necessary condition for a solution of \eqref{eq:RK-fold-red1} to be well-defined is $x^2+ha_{21}x\geq0$.
\end{lemma}
\begin{proof}
	Let $s=2$, then \eqref{eq:RK-fold-red1} reads as $\tilde x^2 = x^2 + h\alpha_1 X_1 + h\alpha_2X_2=x^2 + h\alpha_1 x + h\alpha_2X_2$. Next we have that $X_2 = \pm\sqrt{Y_2}=\pm\sqrt{x^2+ha_{21}X_1}=\pm\sqrt{x^2+ha_{21}x}$, where we already see the required necessity. For $s>2$ the result follows from the fact that \eqref{eq:RK-fold-red1} reads as $\tilde x^2 = x^2 + h\alpha_1 X_1 + h\alpha_2X_2 + \sum_{i=3}^s\alpha_iX_i$ and the value of $X_2$ remains the same.
\end{proof}

The previous lemma implies that, just as in the forward-Euler case, there is a gap $x\in(-ha_{21},0)$ where the solutions of difference equation \eqref{eq:RK-fold-red1} are not well-defined. In other words, again there is no intersection of the critical manifolds $\cS_{\txtr}$ and $\cS_{\txta}$.

Clearly, due to the above exposition, we have to use a different, structure-preserving discretization if we want to understand a canard solution for the fold case in discrete-time. This is another motivation for considering the Kahan-Hirota-Kimura scheme, introduced in the next section.

\section{A-stable methods, Kahan-Hirota-Kimura scheme and symmetrtic loss of stability}\label{sec:Kahan}
\blue{In the following we will demonstrate that certain symmetric methods which preserve stability behaviour are the right choice, for preserving the linear (and also non-trivially) the close nonlinear stability behaviour. In more detail, for ODEs driven by a vector field $f$
we consider, for $a \in \mathbb{R}$, implicit Runge-Kutta methods of the form
        \begin{equation} \label{eq:genscheme_a}
            \frac{\tilde x - x}{h} = a f(x) + (1-2a) f\left( \frac{x + \tilde x}{2}\right) + a f(\tilde x). 
        \end{equation} 
        For example, for $a= \frac{1}{2}$, this is the trapezoid method and for $a=0$ this is the midpoint rule.
        As can be seen easily, these methods are all symmetric, i.e.~time-reversible, A-stable and second order (see \cite{Celledonietal2013}). Recall that A-stability cannot be satisfied by explicit Runge-Kutta methods \cite[Theorem 10.2.7]{ku2015}.}
        
\blue{In particular,} we will consider a method that has produced \emph{integrable} maps in several examples, i.e.~maps that conserve a certain quantity, \blue{and is very suitable for quadratic vector fields where it becomes explicit}: the Kahan-Hirota-Kimura discretization scheme (see e.g.~\cite{PetreraSuris18}) was introduced by Kahan in the unpublished lecture notes \cite{Kahan93} for ODEs with quadratic vector fields
\begin{equation} \label{genODE}
\dot z = f(z) = Q(z) + B z + c \,,
\end{equation}
where each component of $Q: \mathbb R^n \to \mathbb R^n$ is a quadratic form, $B \in \mathbb R^{n \times n}$ and $c \in \mathbb R^n$. The Kahan-Hirota-Kimura discretization, short \emph{Kahan method}, reads as
\begin{equation} \label{genKahan}
\frac{\tilde z - z}{h} = \bar Q(z, \tilde z) + \frac{1}{2} B( z + \tilde z) + c\,,
\end{equation}
where
$$ \bar Q(z, \tilde z)  = \frac{1}{2} ( Q(z +\tilde z) - Q(z) - Q(\tilde z))$$ 
is the symmetric bilinear form such that $ \bar Q(z,z) = Q(z)$. Note that equation~\eqref{genKahan} is linear with respect to $\tilde z$ and by that defines a \emph{rational} map $\tilde z = \Lambda_f (z, h)$, which approximates the time $h$ shift along the soultions of the ODE~\eqref{genODE}. Further note that $\Lambda_f^{-1}(z,h) = \Lambda_f (z, - h)$ and, hence, the map is \emph{birational}. 

The explicit form of the map $\Lambda_f$ defined by equation~\eqref{genKahan} is given as
\begin{equation} \label{Kahanexplicit}
\tilde z = \Lambda_f (z, h) = z + h\left(\Id - \frac{h}{2} \rmD f(z)\right)^{-1} f(z)\,.
\end{equation}
The Kahan method is the specific form, for quadratic vector fields, of an implicit Runge-Kutta scheme \blue{of the form~\eqref{eq:genscheme_a} with $a=-\frac{1}{2}$}, i.e.~given by (cf. \cite[Proposition 1]{Celledonietal2013})
\begin{equation} \label{RungeKutta}
\frac{\tilde z - z}{h} =  - \frac{1}{2} f(z) + 2f\left( \frac{z+\tilde z}{2} \right) - \frac{1}{2} f(\tilde z) \,.
\end{equation}
Note that the ODE~\eqref{ODE_pitchfork_overview}, i.e.~the pitchfork problem, has a cubic term \blue{such that the Kahan method can not be used in its explicit form and the canards are not given as explicit solutions. However, we will use the pitchfork case for demonstrating the dependence of methods of the form on the parameter $a$ and by that discuss the roles of symmetry and A-stability for our discretized canard problems.}
\subsection{Transcritical singularity} \label{sec:kahan_transcritical}
The Kahan discretization of equation~\eqref{ODE_transcrit} gives the map $P_{\textnormal{K}}:\R^2 \setminus \{x = \frac{1}{h}\} \to\R^2$, written as
\begin{equation} \label{Kahan_transcrit}
P_{\textnormal{K}}(x,y) = \begin{pmatrix}[1.7]
\tilde x  \\ \tilde y 
\end{pmatrix}
=
\begin{pmatrix}[1.7]
\dfrac{x + \epsilon h - h y(y + \epsilon h)}{1 - h x} \\
y + \epsilon h
\end{pmatrix},
\end{equation}
where $0<\ve\ll1$, and $h>0$.

Similarly to the continuous-time and the Runge-Kutta case, we find the diagonal to be an invariant curve for~\eqref{Kahan_transcrit} with special canard solution $\gamma$.
In more detail, we have the following statements:
\begin{proposition} \label{prop:Kahan_invariant_transcrit}
The diagonal
\begin{equation} \label{Kahan_invariant_transcrit}
\cD : = \left\{ (x,y) \in \mathbb{R}^2 \, : \, y = x \right\} 
\end{equation}
is invariant under iterations of $P_{\textnormal{K}}$~\eqref{Kahan_transcrit}. 
Solutions on $\cD$ are given by
\begin{equation}  \label{Kahan_gensol_transcrit}
\gamma_{x(0)}(n) = \left( x(0)+\epsilon h n , x(0)+\epsilon h n \right), \forall n \in \mathbb{Z}.
\end{equation}
In particular, for $(x,y)\in \cD$ we have: 
\begin{equation} \label{eq:magnitude_deriv}
\left|\frac{\partial\tilde x}{\partial x}\right|  \quad \left\{ \begin{array}{ll} < 1 & \mathrm{as \;long\; as\;} x < - \epsilon h/2, \\ =1 & \mathrm{ for \;}x=- \epsilon h/2,\\ >1 & \mathrm{as \;long\; as\;} x > - \epsilon h/2, \, x \neq 1/h. \end{array}\right.
\end{equation}
A special canard solution, symmetric with respect to the partition 
$$\cD = S_{\textnormal{a}} \cup \{(-\epsilon h/2, - \epsilon h/2)\} \cup  S_{\textnormal{r}},$$ where
\begin{equation} \label{Kahan_invariant_transcrit_a_r}
S_{\textnormal{a}} = \left\{ (x,y) \in \cD \, : \, x <- \epsilon h/2 \right\}, \ S_{\textnormal{r}} = \left\{ (x,y) \in \cD \, : \, x > - \epsilon h/2 \right\},
\end{equation}
is given for $x(0) = -\epsilon h/2$ and denoted by
\begin{equation} \label{Kahan_specsol_transcrit}
\gamma(n)  = \left(\epsilon h \frac{2n-1}{2}, \epsilon h \frac{2n-1}{2} \right), \forall n \in \mathbb{Z}.
\end{equation} 
\end{proposition}
\begin{proof}
The invariance of $\cD$ follows from an easy calculation.
Furthermore, observe that if $(x,y) \in \cD$, we have
$$ \tilde x = \frac{x -h x^2 + \epsilon h -\epsilon h^2 x}{1 - h x} = \frac{x\left(1 - h x \right) + \epsilon h \left( 1 - hx \right)}{1 - h x} =  x + \epsilon h\,.$$
We compute the Jacobian matrix associated with~\eqref{Kahan_transcrit} as
\begin{equation} \label{KahanJacobian_transcrit}
\frac{\partial (\tilde x, \tilde y)}{\partial (x,y)}   = \begin{pmatrix}[1.7]
\dfrac{1 - h^2 y(y + \epsilon h) + \epsilon h^2}{ (1- h x)^2} & \dfrac{ -2 h x - \epsilon h^2}{ 1- h x} \\
0 & 1
\end{pmatrix}.
\end{equation}
In particular, observe that
\begin{align} \label{Kahan_transcrit_Jacobian}
   \frac{\partial \tilde x}{\partial x} (x,y) \big|_{(x,y) \in \cD} &=  \dfrac{1 - h^2 y(y + \epsilon h) + \epsilon h^2}{ (1- h x)^2} \big|_{(x,y) \in \cD} \nonumber\\
   &= \dfrac{1 - h^2 x^2 - x \epsilon h^3 + \epsilon h^2}{ (1- h x)^2}=: \dfrac{f_h(x)}{g_h(x)} =: J_h(x). 
\end{align}
It is easy to calculate that $J_h(- \epsilon h/2) = 1$. 
Furthermore, observe that $J_h$ is strictly increasing for all $x \neq 1/h$, since we obtain with equation~\eqref{Kahan_transcrit_Jacobian} that
\begin{equation}\label{eq:transcrit-deriv}
    J_h(x)' = \dfrac{f_h'(x)g_h(x) - g_h'(x)f_h(x)}{g_h(x)^2} = \dfrac{h(2+ \epsilon h^2)}{(1-hx)^2} > 0\,.
\end{equation}
%
This shows the claim~\eqref{eq:magnitude_deriv}.

The existence of the special canard $\gamma$~\eqref{Kahan_specsol_transcrit} then follows directly.
\end{proof}
We denote the second entry of the matrix~\eqref{KahanJacobian_transcrit} by
$$ \tilde J_h(x) := \dfrac{ -2 h x - \epsilon h^2}{ 1- h x}\,.$$
Similarly to the Runge-Kutta case discussed in Section~\ref{sec:rk_transcrit}, let $\rho\in O(1)$ be a positive constant and set $x(0)=-\rho$. The variational equation of $P_{\textnormal{K}}^n$ along $\gamma_{- \rho}$ reads
\begin{equation}\label{eq:variational_transcrit_kahan}
	v(n+1)= \begin{pmatrix}[1.5]
		J_h (-\rho + \epsilon h n ) & \tilde J_h (-\rho + \epsilon h n  )\\ 0 & 1
		\end{pmatrix}v(n),
\end{equation}
where $v(n)=(v_1(n),\,v_2(n))\in\R^2$ and $n\in\N$. It follows from an easy calculation that the matrix~\eqref{KahanJacobian_transcrit} has a simple eigenvalue $1$ for the eigenvector $(1,1)^\top$ which is a fixed point of equation~\eqref{eq:variational_transcrit_kahan} and characterises the normal direction along the canard. The local contraction/expansion rate is characterised by the solutions of \eqref{eq:variational_transcrit_kahan} in the transversal (hyperbolic) direction, corresponding to the eigenvector $(1,0)^\top$. Thus, the solution of \eqref{eq:variational_transcrit_kahan} with initial condition $(v_1(0),\,v_2(0))=(1,0)$ is given by $(v_1(n),0)$, where
$$v_1(n) = \prod_{k=0}^n J_h(-\rho + \epsilon h k ) \quad \forall n \geq 0.$$
When $\rho = \epsilon h N + \epsilon h/2$ for some $N \in \mathbb{N}$, then $\gamma_{-\rho} = \gamma$, i.e.~we are on the special canard~\eqref{Kahan_specsol_transcrit}. In this case, we can use the symmetry around  $-\epsilon h/2$, define
\begin{equation} \label{linearisation_map_transcrit}
\begin{split}
v^*(m) &:= \prod_{k=0}^m J_h(\epsilon h \frac{2k-1}{2}) \quad \forall m\geq 0, \\
v^*(m) &:= \prod_{k=m}^0 J_h(\epsilon h \frac{2k-1}{2}) \quad \forall m \leq 0,
\end{split}
\end{equation}
and observe that for all $n \geq N$
\begin{equation} \label{eq:lin_partitioned}
v_1(n) = v^*( -N)v^*(n-N).
\end{equation}

This leads to the following statement concerning contraction and expansion along the canard solutions:
\begin{proposition} \label{prop:main_kahan_transcr}
For any $h, \epsilon > 0$ such that $\frac{1}{\epsilon h^2} + \frac{1}{2} \notin \mathbb{N}$, consider the entry point $x(0)= - \rho < 0$. 
\begin{enumerate}
\item If  $\rho = \epsilon h N + \epsilon h/2$ for some $N \in \mathbb{N}$, then
the way-in/way-out map $\psi_h$ given by
\begin{align*}
 1= \prod_{k=0}^{N + \psi_h(-N)} \left|J_h(-\rho + \epsilon h k )\right| &= \left|v_1(N + \psi_h(-N)) \right|\\
 &= \left|v^*(-N)v^*(\psi_h(-N))\right|,
\end{align*}
is well defined and takes the value 
$$\psi_h(-N) = N.$$
In other words, the accumulated contraction and expansion rates compensate each other in perfect symmetry.
\item If, generally,  $\rho \in (\epsilon h N + \epsilon h/2, \epsilon h (N+1) + \epsilon h/2) $ for some $N \in \mathbb{N}$ such that $\left(\frac{1}{h}, \frac{1}{h} \right) \notin \gamma_{-\rho}$, then the way-in/way-out map $\psi_h(-N)$ is given by the smallest natural number such that 
$$ 1\leq  \prod_{k=0}^{N + \psi_h(-N)} \left|J_h(-\rho + \epsilon h k )\right| = \left|v_1(N + \psi_h(-N)) \right|,$$
and satisfies 
\begin{equation}
\psi_h(-N) \in \{ N+1, N+2\}.
\end{equation}

\end{enumerate}
Summarising both cases, we conclude that the expansion has compensated for contraction at 
$$x^* \in \{\rho - \epsilon h, \rho, \rho + \epsilon h\},$$
giving full symmetry of the entry-exit relation.
\end{proposition}
\begin{proof}
It follows from a tedious but straight-forward calculation  that $$J_h\left(\epsilon h \frac{2n-1}{2}\right)J_h\left(\epsilon h \frac{-2n-1}{2}\right) = 1$$
for all $ \epsilon h \frac{2n-1}{2} \neq 1/h$, $n \in \mathbb{Z}$. Hence, the first claim follows immediately.

For the second claim: firstly, recall from \eqref{eq:transcrit-deriv} that $J_h$ is strictly increasing for all $x \neq 1/h$.
Hence, we can estimate
\begin{equation*}
\prod_{k=0}^{N + N+2} \left|J_h(-\rho + \epsilon h k )\right| \geq v_1(N)v^*(N+1) \geq  v^*(-(N+1))v^*(N+1) =1.
\end{equation*}
Therefore we can deduce $\psi_h(-N) \leq N+2$. Furthermore, we get directly from the strictly monotonic increase of $J_h$ that
\begin{equation*}
\prod_{k=0}^{N + N} \left|J_h(-\rho + \epsilon h k )\right| < v^*(-N)v^*(N) =1,
\end{equation*}
such that $\psi_h(-N) > N$ follows. This finishes the proof.
\end{proof}
\begin{remark}
Contraction and expansion balance out completely along the canard solution for the Kahan map which mirrors exactly the time-continuous case where
$$v^*(\rho) = \int_0^{\rho} 2 x \, \rmd x,  $$
such that the way-in/way-out map $\psi$ satisfies $\psi(\rho) = - \rho$ for all $\rho \in I \subset \mathbb{R}$ for some appropriate interval $I$.
We can make the starting and end point of the time-continuous and time-discrete system coincide exactly when simply choosing $\rho = \epsilon h \frac{2n-1}{2}$ for some $n \in \mathbb{N}$..
\end{remark}



\subsection{Pitchfork singularity}
\blue{We now consider the implicit Runge Kutta discretizations of equation~\eqref{ODE_pitchfork_overview}, according to the scheme~\eqref{eq:genscheme_a}. Note that, due to the cubic term, the explicit Kahan discretization~\eqref{genKahan} is not applicable.}

        \blue{The equations read 
\begin{equation} \label{eq:genscheme_pitchfork}
\begin{split} 
\dfrac{\tilde x - x}{h} &= a\left( xy - x^3\right) + (1-2a) \left[ \dfrac{x+\tilde x}{2}\dfrac{y + \tilde y}{2} - \dfrac{(x+ \tilde x)^3}{8}\right] + a\left( \tilde x \tilde y - \tilde x^3\right) , \\
\dfrac{\tilde y - y}{h} &=  \epsilon \,,
\end{split}
\end{equation}
and, e.g.,} the Kahan scheme gives the equations
\begin{equation} \label{eq:Kahan_pitchfork}
\begin{split} 
\dfrac{\tilde x - x}{h} &= - \frac{1}{2}\left( xy - x^3\right) + 2 \left[ \dfrac{x+\tilde x}{2}\dfrac{y + \tilde y}{2} - \dfrac{(x+ \tilde x)^3}{8}\right] - \frac{1}{2}\left( \tilde x \tilde y - \tilde x^3\right) , \\
\dfrac{\tilde y - y}{h} &=  \epsilon \,. 
\end{split}
\end{equation}
As opposed to the transcritical case, we do not obtain an explicit map but possibly several solutions.
\blue{Starting with $x=0$ and $y$ negative, then $\tilde x = 0$ is a solution as well as
        $$ \tilde x^2 = \frac{4-2 hy - h^2 \epsilon (1+2a)}{-h\left(\frac{1+6a}{2} \right)},$$
        as long as the right hand side is nonnegative. 
        For example, for the Kahan method this means that} starting with $x = 0$ and $ y\in (- \infty, 2/h)$, we get
$$ \tilde x \in \left\{ - \sqrt{\dfrac{4 -2 hy}{h}}, 0, \sqrt{\dfrac{4 -2 hy}{h}} \right\} $$
from equation~\eqref{eq:Kahan_pitchfork}. 
We can still analyze a canard solution in the following way:
\begin{proposition} \label{prop:Kahan_invariant_pitchfork}
The set
\begin{equation} \label{Kahan_invariant_pitchfork}
\cY : = \left\{ (x,y) \in \mathbb{R}^2 \, : \, x= 0 \right\} 
\end{equation}
is invariant under particular solutions of \blue{equations~\eqref{eq:genscheme_pitchfork}, and in particular~\eqref{eq:Kahan_pitchfork}} which are given by
\begin{equation}  \label{Kahan_gensol_pitchfork}
\gamma_{ y(0)}(n) = \left( 0, y(0)+\epsilon h n \right), \forall n \in \mathbb{Z}.
\end{equation}
In particular, for $(x,y)\in \cY$ we have \blue{for all $a < \frac{2}{h^2 \epsilon}$, including the Kahan method $a=-\frac{1}{2}$,
\begin{equation} \label{eq:magnitude_deriv_pitchfork}
\left|\frac{\partial\tilde x}{\partial x}\right|  \quad \left\{ \begin{array}{ll} < 1 & \mathrm{as \;long\; as\;} y < - \epsilon h/2, \\ =1 & \mathrm{ for \;}y=- \epsilon h/2,\\ >1 & \mathrm{as \;long\; as\;} y > - \epsilon h/2, \, y \neq \frac{2}{h} - \frac{h(1+2a)}{2} \epsilon. \end{array}\right.
\end{equation}
For $a > \frac{2}{h^2 \epsilon}$, the stability properties are precisely reversed.
When $a < \frac{2}{h^2 \epsilon}$,} a special canard solution, symmetric with respect to the partition 
$$\cY = S_{\textnormal{a}} \cup \{(0, - \epsilon h/2)\} \cup  S_{\textnormal{r}},$$ where
\begin{equation} \label{Kahan_invariant_pitchfork_a_r}
S_{\textnormal{a}} = \left\{ (x,y) \in \cY \, : \, y <- \epsilon h/2 \right\}, \ S_{\textnormal{r}} = \left\{ (x,y) \in \cY \, : \, y > - \epsilon h/2 \right\},
\end{equation}
is given for $y(0) = -\epsilon h/2$ and denoted by
\begin{equation} \label{Kahan_specsol_pitch}
\gamma(n)  = \left(0, \epsilon h \frac{2n-1}{2} \right), \forall n \in \mathbb{Z}.
\end{equation} 
\end{proposition}
\begin{proof}
The existence of trajectories $\gamma_{y(0)}$ on $\cY$ follows from an easy calculation.
Furthermore, observe that, \blue{for any $a \in \mathbb{R}$,
\begin{align*}
\frac{1}{h} \dfrac{\partial \tilde x}{\partial x} &= \frac{1}{h} +a y -3 a x^2 + \frac{1-2a}{4}(y + \tilde y) + \frac{1-2a}{4}(y + \tilde y)\dfrac{\partial \tilde x}{\partial x} \\
&- \dfrac{1-2a}{4}  H(x, \tilde x) +a \tilde y \dfrac{\partial \tilde x}{\partial x} - 3 a \tilde x^2 \dfrac{\partial \tilde x}{\partial x},
\end{align*}
where $H(x, \tilde x) = \mathcal O\left( x^2, \tilde x^2, x \tilde x \right)$.
Hence, at $x= \tilde x = 0$, we obtain
\begin{equation*}
 \dfrac{\partial \tilde x}{\partial x} \left(1- h\dfrac{1-2a}{4}y - h\dfrac{1+2a}{4}(y+ \epsilon h\right) = 1 + h \left(\dfrac{1+2a}{4}y + \dfrac{1-2a}{4}(y+ \epsilon h)\right) ,
\end{equation*}
which gives
\begin{equation} \label{genschem_a_pitch_Jacobian}
   \frac{\partial \tilde x}{\partial x} (x,y) \big|_{(x,y) \in \cY}  = \dfrac{1+ \frac{h}{2}y + \frac{h^2(1-2a)}{4}\epsilon}{1- \frac{h}{2}y - \frac{h^2(1+2a)}{4}\epsilon} =: J_{h,a}(y). 
\end{equation} 
Similarly to the transcritical case, one can compute that $J_{h,a}(- \epsilon h/2) = 1$ for all choices of $a \in \mathbb R \setminus \{\frac{2}{h^2 \epsilon}\}$,  and for all $a \in \mathbb R, y \neq \frac{2}{h}- \frac{h(1+2a)}{2} \epsilon$, we have
$$
J_{h,a}'(y) = \dfrac{h - \frac{h^3 \epsilon a}{2}}{\left( 1- \frac{h}{2}y - \frac{h^2(1+2a)}{4}\epsilon\right)^2} \begin{cases} > 0 \ &\text{ if } a < \frac{2}{h^2 \epsilon},\\
=0 \ &\text{ if } a =\frac{2}{h^2 \epsilon}
\\
< 0 \ &\text{ if } a > \frac{2}{h^2 \epsilon}.\end{cases}
$$
Hence, in the first case stability behavior is preserved, and in the third case reversed, both situations separated by the critical value $a=\frac{2}{h^2 \epsilon} $.
}
In the case of the Kahan method we obtain
\begin{equation} \label{Kahan_pitch_Jacobian}
   \frac{\partial \tilde x}{\partial x} (x,y) \big|_{(x,y) \in \cY}  = \dfrac{1+ \frac{h}{2}y + \frac{h^2}{2}\epsilon}{1- \frac{h}{2}y} =: J_{h}(y).
\end{equation} 
and
$$
J_h'(y) = \dfrac{h + \frac{h^3 \epsilon}{4}}{\left( 1- \frac{h}{2}y\right)^2} > 0 , \ \forall y \neq \frac{2}{h}.
$$
The existence of the special canard $\gamma$~\eqref{Kahan_specsol_pitch} then follows directly.
\end{proof}
%
\blue{\begin{remark} \label{rem:Astability}
Proposition~\ref{prop:Kahan_invariant_pitchfork} shows that not any A-stable method preserves exactly the continuous-time behavior in terms of stable and unstable manifolds but only in certain cases. In fact, here, the stability properties can be exactly reversed. Note that this has to do with the fact that the discretization of the linearization and the linearization of the discretization do not necessarily commute. Hence, also A-stable methods can fail in preserving stability behaviour.
In more detail, if we consider the linearization of the ODE~\eqref{ODE_pitchfork_overview}, given along $\cY$ as $v'= y v$ with $y' = \epsilon$, then any method of the form~\eqref{eq:genscheme_a} gives
\begin{equation*}
    \tilde v = v \dfrac{4 + (\tilde y +y)h}{4 - (\tilde y +y)h}, \quad \tilde y = y + \epsilon h.
\end{equation*}
Hence, stability behavior is preserved on both sides of $y= - \epsilon h/2$ for any choice of $h$, due to A-stability independently from $a$. But we want to understand the behavior of the nonlinear maps as discretizations of the nonlinear ODE~\eqref{ODE_pitchfork_overview}, by the help of linearization of the maps along special trajectories, in this case canards. The results summarized in Proposition~\ref{prop:Kahan_invariant_pitchfork} demonstrate that the problem is much more subtle than covered by the concept of A-stability.
\end{remark}}
\blue{The following is formulated without loss of generality for the Kahan method but, of course, can be extended to all cases $a < \frac{2}{h^2 \epsilon}$.}

Similarly to the transcritical case, let $\rho\in O(1)$ be a positive constant and set $y(0)=-\rho$. It is easy to compute
$$ \frac{\partial \tilde x}{\partial y} (x,y) \big|_{(x,y) \in \cY} = 0.$$
Hence, the variational equation along $\gamma_{-\rho}$ reads
\begin{equation}\label{eq:variational_pitch_kahan}
	v(n+1)= \begin{pmatrix}[1.5]
		J_h (-\rho + \epsilon h n ) & 0 )\\ 0 & 1
		\end{pmatrix}v(n),
\end{equation}
where $v(n)=(v_1(n),\,v_2(n))\in\R^2$ and $n\in\N$. The Jacobian matrix has a simple eigenvalue $1$ for the eigenvector $(0,1)^\top$ which is a fixed point of equation~\eqref{eq:variational_pitch_kahan} and characterises the normal direction along the canard. The local contraction/expansion rate is characterised by trajectories in the transversal (hyperbolic) direction, corresponding to the eigenvector $(1,0)^\top$. Thus, the solution of \eqref{eq:variational_pitch_kahan} with intial condition $(v_1(0),\,v_2(0))=(1,0)$ is given by $(v_1(n),0)$, where
$$v_1(n) = \prod_{k=0}^n J_h(-\rho + \epsilon h k ) \quad \forall n \geq 0.$$
When $\rho = \epsilon h N + \epsilon h/2$ for some $N \in \mathbb{N}$, then $\gamma_{- \rho} = \gamma$, i.e.~we are on the special canard~\eqref{Kahan_specsol_pitch}. As before, we introduce
\begin{equation} \label{linearisation_map_pitch}
\begin{split}
v^*(m) &:= \prod_{k=0}^m J_h(\epsilon h \frac{2k-1}{2}) \quad \forall m\geq 0, \\
v^*(m) &:= \prod_{k=m}^0 J_h(\epsilon h \frac{2k-1}{2}) \quad \forall m \leq 0,
\end{split}
\end{equation}
and observe that for all $n \geq N$
\begin{equation} \label{eq:lin_partitioned_pitch}
v_1(n) = v^*( -N)v^*(n-N).
\end{equation}

This leads to the following statement, analogously to Proposition~\ref{prop:main_kahan_transcr}:
\begin{proposition} \label{prop:main_kahan_pitch}
For any $h, \epsilon > 0$ such that $\frac{2}{\epsilon h^2} + \frac{1}{2} \notin \mathbb{N}$, consider the entry point $y(0)= - \rho < 0$. 
\begin{enumerate}
\item If  $\rho = \epsilon h N + \epsilon h/2$ for some $N \in \mathbb{N}$, then
the way-in/way-out map $\psi_h$ given by
\begin{align*}
 1= \prod_{k=0}^{N + \psi_h(-N)} \left|J_h(-\rho + \epsilon h k )\right| &= \left|v_1(N + \psi_h(-N)) \right|\\
 &= \left|v^*(-N)v^*(\psi_h(-N))\right|,
\end{align*}
is well defined and takes the value 
$$\psi_h(-N) = N.$$
In other words, the accumulated contraction and expansion rates compensate each other in perfect symmetry.
\item If, generally,  $\rho \in (\epsilon h N + \epsilon h/2, \epsilon h (N+1) + \epsilon h/2) $ for some $N \in \mathbb{N}$ such that $\left(0, \frac{2}{h} \right) \notin \gamma_{-\rho}$, then the way-in/way-out map $\psi_h(-N)$ is given by the smallest natural number such that 
$$ 1\leq  \prod_{k=0}^{N + \psi_h(-N)} \left|J_h(-\rho + \epsilon h k )\right| = \left|v_1(N + \psi_h(-N)) \right|,$$
and satisfies 
\begin{equation}
\psi_h(-N) \in \{ N+1, N+2\}.
\end{equation}

\end{enumerate}
Summarising both cases, we conclude that the expansion has compensated for contraction at 
$$y^* \in \{\rho - \epsilon h, \rho, \rho + \epsilon h\},$$
giving full symmetry of the entry-exit relation.
\end{proposition}
\begin{proof}
Similarly to the transcritical case, we observe that
 $$J_h\left(\epsilon h \frac{2n-1}{2}\right)J_h\left(\epsilon h \frac{-2n-1}{2}\right) = 1$$
for all $ \epsilon h \frac{2n-1}{2} \neq 2/h$, $n \in \mathbb{Z}$. Hence, the first claim follows immediately.
The second claim can be deduced from arguments analagously to the proof of Proposition~\ref{prop:main_kahan_transcr}.
\end{proof}
In the transcritical and pitchfork case, the canards are given on lines and we observe the same linear strcuture for the Kahan discretizations in both situatiuons, with exactly the same symmetry properties. We will turn to the problem of a folded canard in the next subsection.
\subsection{Fold singularity}\label{sec:K-fold}
In contrast to~\cite{Engeletal2019}, we will restrict the following analysis to the most basic canonical form of a fast-slow system with fold singularity, i.e.~model~\eqref{ODE_fold}, since we are only interested in the properties of the linearization along the canard solution.

The Kahan discretization of system~\eqref{ODE_fold} reads as
\begin{equation}
\frac{\tilde x - x}{h} = \tilde x x - \frac{\tilde y + y}{2}, \quad \frac{\tilde y - y}{h} = \epsilon \frac{\tilde x + x}{2} , 
\end{equation}
and induces the invertible, birational map $P_{\textnormal{K}}: \mathbb{R}^2 \to \mathbb{R}^2$,  written explicitly as
\begin{equation} \label{map_Kahan}
P_{\textnormal{K}}: \begin{pmatrix}[2.5]
x \\ y 
\end{pmatrix}
\mapsto \begin{pmatrix}[2.5]
\tilde{x}  \\ \tilde y 
\end{pmatrix}
= \begin{pmatrix}[2.5]
\dfrac{x - hy - \frac{h^2}{4} \epsilon x}{ 1- hx + \frac{h^2}{4} \epsilon} \\ 
\dfrac{y - hyx - \frac{h^2}{2} \epsilon x^2  +  h \epsilon x - \frac{h^2}{4} \epsilon y}{ 1- hx + \frac{h^2}{4} \epsilon}
\end{pmatrix}.
\end{equation}
Similarly to before, we make the following observations (which can be found in a similar form in \cite{Engeletal2019}):
\begin{proposition} \label{prop:Kahan_invariant_fold}
The parabola
\begin{equation} \label{Kahan_invariant}
\cS_{\ve} : = \left\{ (x,y) \in \mathbb{R}^2 \, : \, y = x^2 - \frac{\epsilon}{2} - \frac{\epsilon^2 h^2}{8} \right\} 
\end{equation}
is invariant under iterations of $P_{\textnormal{K}}$~\eqref{map_Kahan}. Solutions on $\cS_{\ve}$ are given by
\begin{equation} \label{Kahan_gensol_fold}
\gamma_{x(0)}(n) = \left( x(0)+  n\frac{\epsilon h}{2},\left( x(0)+  n\frac{\epsilon h}{2}\right)^2 - \frac{\epsilon}{2} - \frac{\epsilon^2 h^2}{8} \right), \forall n \in \mathbb{Z}.
\end{equation}
For $(x,y)\in \cS_{\ve}$ we have
\begin{equation}
\left|\frac{\partial\tilde x}{\partial x}\right|  \quad \left\{ \begin{array}{ll} < 1 & \mathrm{as \;long\; as\;} x < 0, \\ =1 & \mathrm{ for \;}x=0,\\ >1 & \mathrm{as \;long\; as\;} x > 0, \, x \neq \left(1+\frac{h^2\epsilon}{4}\right)/h. \end{array}\right.
\end{equation}
A special canard solution, symmetric with respect to the partition 
$$\cS_{\ve} = \cS_{\textnormal{a},\ve} \cup \{(0, 0)\} \cup  \cS_{\textnormal{r},\ve},$$ where
\begin{equation} \label{Kahan_invariant_fold_a_r}
\cS_{\textnormal{a},\ve} = \left\{ (x,y) \in \cS_\ve \, : \, x <0 \right\}, \ \cS_{\textnormal{r},\ve} = \left\{ (x,y) \in \cS_\ve \, : \, x > 0 \right\},
\end{equation}
is given for $x(0) = 0$ and denoted by
\begin{equation} \label{Kahan_specsol_fold}
\gamma(n) = \left( n\frac{\epsilon h}{2},\left(n\frac{\epsilon h}{2}\right)^2 - \frac{\epsilon}{2} - \frac{\epsilon^2 h^2}{8} \right), \forall n \in \mathbb{Z}.
\end{equation}
\end{proposition}
\begin{proof}
The invariance of $\cS_{\ve}$ follows from a cumbersome but straight-forward calculation, checking that for 
$$ y = x^2 - \frac{\epsilon}{2} - \frac{\epsilon^2 h^2}{8}, $$
we indeed have
$$ \tilde y = \tilde x^2 - \frac{\epsilon}{2} - \frac{\epsilon^2 h^2}{8}.$$
Furthermore, observe that if $(x(0),y(0)) \in \cS_{\ve}$, we have
$$ \tilde x = \frac{x -h x^2 + \frac{\epsilon h}{2} + \frac{\epsilon^2 h^3}{8} - \frac{h^2 \epsilon }{4}x}{1 - h x + \frac{h^2}{4}\epsilon} = \frac{\left(1 - h x + \frac{h^2}{4}\epsilon\right)\left( x + \frac{h\epsilon}{2}\right)}{1 - h x + \frac{h^2}{4}} =  x + \frac{h \epsilon}{2}\,,$$
which shows the existence of $\gamma_{x(0)}$~\eqref{Kahan_gensol_fold}.

We compute the Jacobian matrix associated with~\eqref{map_Kahan} as
\begin{equation} \label{KahanJacobian}
\frac{\partial (\tilde x, \tilde y)}{\partial (x,y)}   = \begin{pmatrix}[2]
\frac{1-h^2y - \frac{h^4\epsilon^2}{16}}{\left(1-hx+\frac{h^2\epsilon}{4} \right)^2} & - \frac{h}{1-hx+\frac{h^2 \epsilon}{4}} \\
\frac{h\epsilon -h^2\epsilon x + \frac{h^3\epsilon}{4}(2x^2-2y+\epsilon) -\frac{h^4\epsilon^2}{4}x}{\left(1-hx+\frac{h^2\epsilon}{4} \right)^2} & \frac{1-hx-\frac{h^2\epsilon}{4}}{1-hx+\frac{h^2\epsilon}{4}}
\end{pmatrix}.
\end{equation}
In particular, observe that for $(x,y)\in \cS_{\ve}$ we have
\begin{equation} \label{Kahan_fold_Jacobian}
\frac{\partial \tilde x}{\partial x} (x,y) =  \dfrac{- h^2 x^2 + \left(1+\frac{h^2 \epsilon}{4}  \right)^2}{ \left(1-hx+\frac{h^2\epsilon}{4} \right)^2} =: J_h(x)\,.
\end{equation}
Clearly, $J_h(0) =1$. Moreover, we observe for all $x \in \mathbb{R} \setminus \left\{\left(1+\frac{h^2\epsilon}{4}\right)/h\right\}$ that
$$ J_h'(x) = \dfrac{2 h (1+ \frac{h^2\epsilon}{4})}{ \left(1-hx+\frac{h^2\epsilon}{4} \right)^2} > 0\,.$$
This concludes the claim.
\end{proof}
Similarly to the transcritical and the pitchfork case, let $\rho\in O(1)$ be a positive constant and set $x(0)=-\rho$. 
Similar to what we have done for the other singularities, we are going to consider the variational equation along $\gamma_{ - \rho}$ only in the $x$-direction. Note that the only point along $\gamma_{ - \rho}$ that is tangent to a horizontal line is at $p_0=(x,y)=\left( 0, - \frac{\epsilon}{2} - \frac{\epsilon^2 h^2}{8}\right)$. However, it follows from \eqref{KahanJacobian} that $\frac{\partial\tilde x}{\partial x}(p_0)=1$, meaning that at $p_0$ there is no contraction nor expansion. This observation indeed allows us to only focus on the factors $\frac{\partial \tilde x}{\partial x} (x,y)$ as contraction/expansion rates giving
$$v_1(n) = \prod_{k=0}^n J_h \left(-\rho + \frac{\epsilon h k}{2} \right) \quad \forall n \geq 0.$$
When $\rho = \epsilon h N/2$ for some $N \in \mathbb{N}$, then $\gamma_{ - \rho} = \gamma$, i.e.~we are on the special canard~\eqref{Kahan_specsol_fold}. As before, we introduce
\begin{equation} \label{linearisation_map_fold}
\begin{split}
v^*(m) &:= \prod_{k=0}^m J_h \left( \frac{\epsilon h k}{2} \right) \quad \forall m\geq 0, \\
v^*(m) &:= \prod_{k=m}^0 J_h \left(\frac{\epsilon h k}{2} \right) \quad \forall m \leq 0,
\end{split}
\end{equation}
and observe that for all $n \geq N$
\begin{equation} \label{eq:lin_partitioned_fold}
v_1(n) = v^*( -N)v^*(n-N).
\end{equation}
This leads to the following statement, analogously to Proposition~\ref{prop:main_kahan_transcr} and Proposition~\ref{prop:main_kahan_pitch}:
\begin{proposition} \label{prop:main_kahan_fold}
For any $h, \epsilon > 0$ such that $\frac{2}{\epsilon h^2} + \frac{1}{2} \notin \mathbb{N}$, consider the entry point $x(0)= - \rho < 0$. 
\begin{enumerate}
\item If  $\rho = \epsilon h N/2$ for some $N \in \mathbb{N}$, then
the way-in/way-out map $\psi_h$ given by
\begin{align*}
 1= \prod_{k=0}^{N + \psi_h(-N)} \left|J_h \left(-\rho + \frac{\epsilon h k}{2} \right)\right| &= \left|v_1(N + \psi_h(-N)) \right|\\
 &= \left|v^*(-N)v^*(\psi_h(-N))\right|,
\end{align*}
is well defined and takes the value 
$$\psi_h(-N) = N.$$
In other words, the accumulated contraction and expansion rates compensate each other in perfect symmetry.
\item If, generally,  $\rho \in (\epsilon h N/2, \epsilon h (N+1)/2) $ for some $N \in \mathbb{N}$ such that $\left(1+\frac{h^2\epsilon}{4}\right)/h\neq \gamma_{-\rho}^1(n)$ for all $n \in \mathbb{N}$ (where $\gamma_{-\rho}^1(n)$ denotes the first component of $\gamma_{-\rho}(n)$),  then the way-in/way-out map $\psi_h(-N)$ is given by the smallest natural number such that 
$$ 1\leq  \prod_{k=0}^{N + \psi_h(-N)} \left|J_h \left(-\rho + \frac{\epsilon h k}{2} \right)\right| = \left|v_1(N + \psi_h(-N)) \right|,$$
and satisfies 
\begin{equation}
\psi_h(-N) \in \{ N+1, N+2\}.
\end{equation}

\end{enumerate}
Summarising both cases, we conclude that the expansion has compensated for contraction at 
$$x^* \in \{\rho - \epsilon h, \rho, \rho + \epsilon h\},$$
giving full symmetry of the entry-exit relation.
\end{proposition}
\begin{proof}
We show that $J_h(x)J_h(-x) = 1$ for all $ x \neq (1+h^2/4)/h$. Then the first claim follows immediately.
\begin{align*}
 J_h(x)J_h(-x) &=   \\
 &=\dfrac{- h^2 x^2 + \left(1+\frac{h^2}{4}  \right)^2}{ h^2 x^2 - 2\left(x h + \frac{h^3}{4}x\right) + \left(1+\frac{h^2}{4} \right)^2} \dfrac{- h^2 x^2 + \left(1+\frac{h^2}{4}  \right)^2}{ h^2 x^2 + 2\left(x h + x\frac{h^3}{4}\right) + \left(1+\frac{h^2}{4} \right)^2} \\
 &= \dfrac{ h^4 x^4 -2 \left(h x+\frac{h^3}{4}x  \right)^2 +\left(1+\frac{h^2}{4}  \right)^4}{ h^4 x^4 -4 \left(h x+\frac{h^3}{4}x  \right)^2 + 2 \left(h x+\frac{h^3}{4}x  \right)^2 +\left(1+\frac{h^2}{4} \right)^4} = 1.
\end{align*}
The second claim can be deduced from arguments analogously to the proof of Proposition~\ref{prop:main_kahan_transcr} and Proposition~\ref{prop:Kahan_invariant_pitchfork}.
\end{proof}
\begin{remark}
Again, contraction and expansion balance out completely along the canard solution for the Kahan map which mirrors exactly the time-continuous case where
$$v^*(\rho) = \int_0^{\rho} 2 x \, \rmd x,  $$
such that the way-in/way-out map $\psi$ satisfies $\psi(\rho) = - \rho$ for all $\rho \in I \subset \mathbb{R}$ for some appropriate interval $I$.
We can make the starting and end point of the time-continuous and time-discrete system coincide exactly when simply choosing $\rho =  \frac{\epsilon hn}{2}$ for some $n \in \mathbb{N}$.
\end{remark}

We conclude this section by noting that Theorem \ref{thm:Kahan} immediately follows from Propositions \ref{prop:Kahan_invariant_transcrit}, \ref{prop:main_kahan_transcr}, \ref{prop:Kahan_invariant_pitchfork},  \ref{prop:main_kahan_pitch}, \ref{prop:Kahan_invariant_fold}, \ref{prop:main_kahan_fold}.


\section{Numerical simulations}\label{sec:numerics}

\blue{

In this section we are going to support and validate our analysis presented above by a series of numerical simulations. All the forthcoming simulations are concerned with the transcritical singularity and have been performed with Python version 3.8.1 (this becomes relevant below due to the numerical precision one needs to take into account for performing simulations of fast-slow systems).

To start, we present in Figure \ref{fig:critical-triplet} surfaces that depict the critical triplets $(\rho^*, h^*, \ve^* )$ (see Definition~\ref{def:crit-triplet}) for five different numerical schemes: the Euler method and four common third order RK-schemes. We emphasize that we do not claim that a critical triplet exists for every explicit RK scheme.

\begin{figure}[htbp]\centering
    \includegraphics[scale=1]{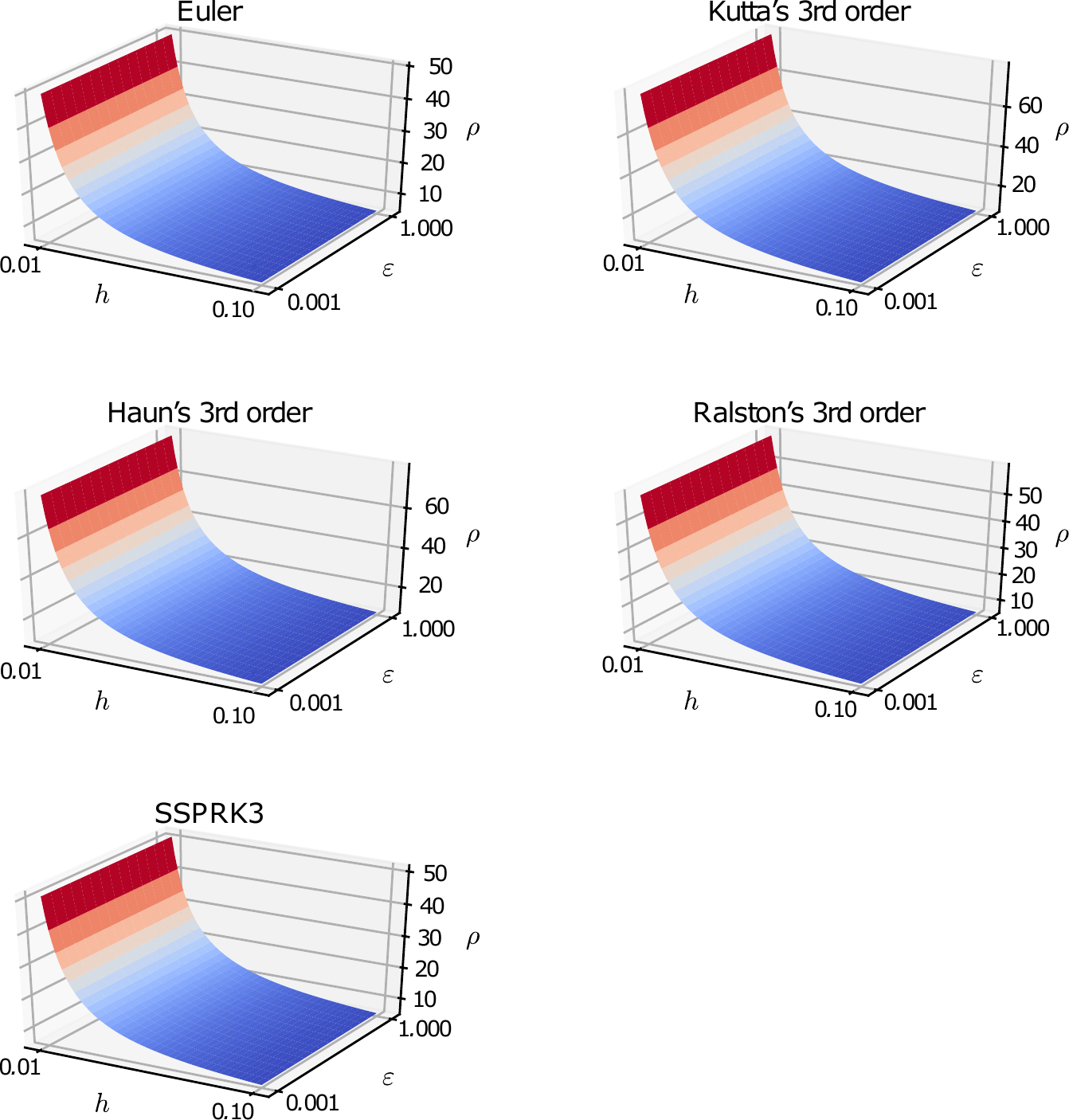}
    \caption{\blue{Surfaces of critical triplets for transcritical singularities and several explicit RK-methods.}}
    \label{fig:critical-triplet}
\end{figure}

\clearpage

The surfaces of Figure \ref{fig:critical-triplet} show the values of the critical triplets (see Definition \ref{def:crit-triplet}) for a fast-slow system with a transcritical singularity. The plot on the upper-left corner corresponds to the Euler method, which is an explicit RK-method of order $1$. All the other plots in  Figure \ref{fig:critical-triplet} correspond to common explicit RK-methods whose names appear at the top of each plot. SSPRK3 stands for ``Strong Stability Preserving Runge-Kutta of order $3$''. One observes from Figure \ref{fig:critical-triplet} that, although different explicit RK methods may be used, the values of the critical triplets are approximately the same. More importantly, the plots illustrate that, at least for the transcritical singularity, the variation of the critical value $\ve^*$ is vanishingly small compared to changes of the critical value $h^*$; as we can see in Figure \ref{fig:critical-triplet}, the critical triplet appears to be constant along the $\ve$-axis (although in detail one can see slight variations). 

We note the following critical issue regarding computer simulations --- which are performed in Python in this case, but the arguments transfer naturally to any programming language:

\begin{remark} \label{rem:num_issues}
Canards are objects that are very difficult to track numerically. In particular for the transcritical singularity (without higher order terms) the maximal canard is the diagonal $\mathcal{D}=\left\{ x=y \right\}$. We recall that $\mathcal D$ is invariant for the continuous time system but also for all maps obtained by an explicit RK discretization scheme. 

This causes a serious issue when simulating a fast-slow system with a transcritical singularity: the default floating point error in common modern computers is $\frac{1}{2^{23}}\approx1\times 10^{-16}$. Besides the potential risk of errors being accumulated after a large number of iterations, two distinct numbers being apart by less than $1\times 10^{-16}$ may be interpreted as the same number, especially if one does not control the approximation error of the numerical computations. This is particularly inconvenient for the transcritical singularity given the fact that trajectories get exponentially close to the diagonal $\mathcal D$. We exemplify this numerical issue in the left panel of Figure \ref{fig:sticky}, where we observe that the shown trajectory does not leave the invariant diagonal $\mathcal D$ due to the fact that at some point it lies within distance less than $1\times 10^{16}$ from $\mathcal D$. 

\begin{figure}[htbp]
    \centering
    \begin{tikzpicture}
    \node at (0,0) {
    \includegraphics{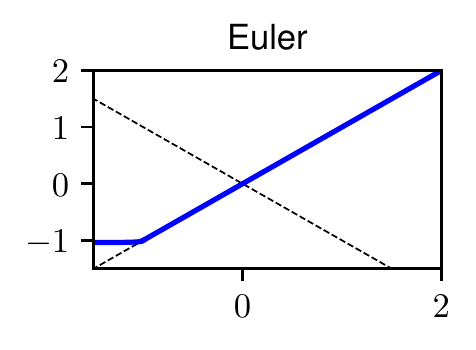}
    };
    \node at (0,-2) {$x$};
    \node at (-2.5,0) {$y$};
    \end{tikzpicture}
    \hspace*{2cm}
    \begin{tikzpicture}
    \node at (0,0) {
    \includegraphics{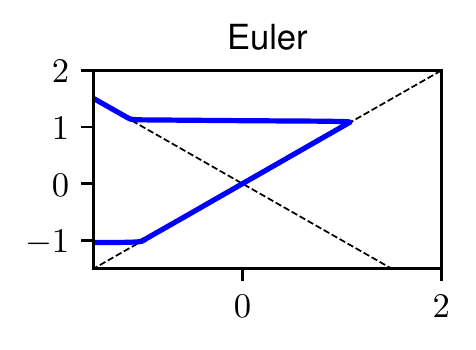}
    };
    \node at (0,-2) {$x$};
    \node at (-2.5,0) {$y$};
    \end{tikzpicture}
    \caption{\blue{Simulations of a fast-slow system with a transcritical singularity using the Euler method (the dashed black lines correspond to the critical set $\left\{ x^2=y^2\right\}$). Both simulations show an orbit passing through the point $(x,y)=(-1,-1+ 10^{-4})$. Moreover, both simulations are for values $(h,\ve)=(10^{-4},10^{-2})$. On the left panel we show the simulation corresponding to default settings, i.e., with a floating point error of approx. $1\times 10^{-16}$. On the right panel we show the same simulation but now with a working precision of $50$ decimal digits. We notice that, even though the time-step $h$ is small, only the right panel displays the expected result. We also remark that this issue is independent of the discretization method.}}
    \label{fig:sticky}
\end{figure}
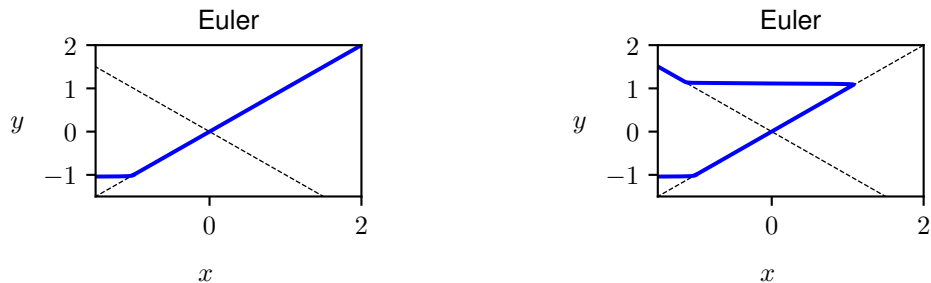

To overcome the aforementioned numerical problem, we are using the Python library {\texttt{mpmath}}~\cite{mpmath} which allows one to choose an arbitrary working precision for the simulation. To compare we are showing in the right panel of Figure \ref{fig:sticky} a simulation with exactly the same settings, but with a working precision of $50$ decimal digits. Setting a high enough floating point precision for our numerical simulations becomes even more important as we approach parameter values $(\rho,h,\ve)$ close to the critical triplet. Therefore, all our forthcoming simulations have been performed with a decimal precision of $5000$ decimal digits. This means that two floating point numbers are considered to be equal only if they coincide up to their $5000th$ decimal digit.

\end{remark}

To showcase numerically the long delays that one may observe in simulations, we present in Figure \ref{fig:delays} a series of plots for several values of the critical triplet. For these plots we have chosen values of critical triplets according to Figure \ref{fig:critical-triplet}. However one should note that such values are computed from the linearization of the discretized system, and thus differ slightly from the true critical triplet of the nonlinear system. Therefore, we provide in Table \ref{tab:values} more precise values of the critical triplets considered for our simulations, and in particular of the step size $h^*$. 

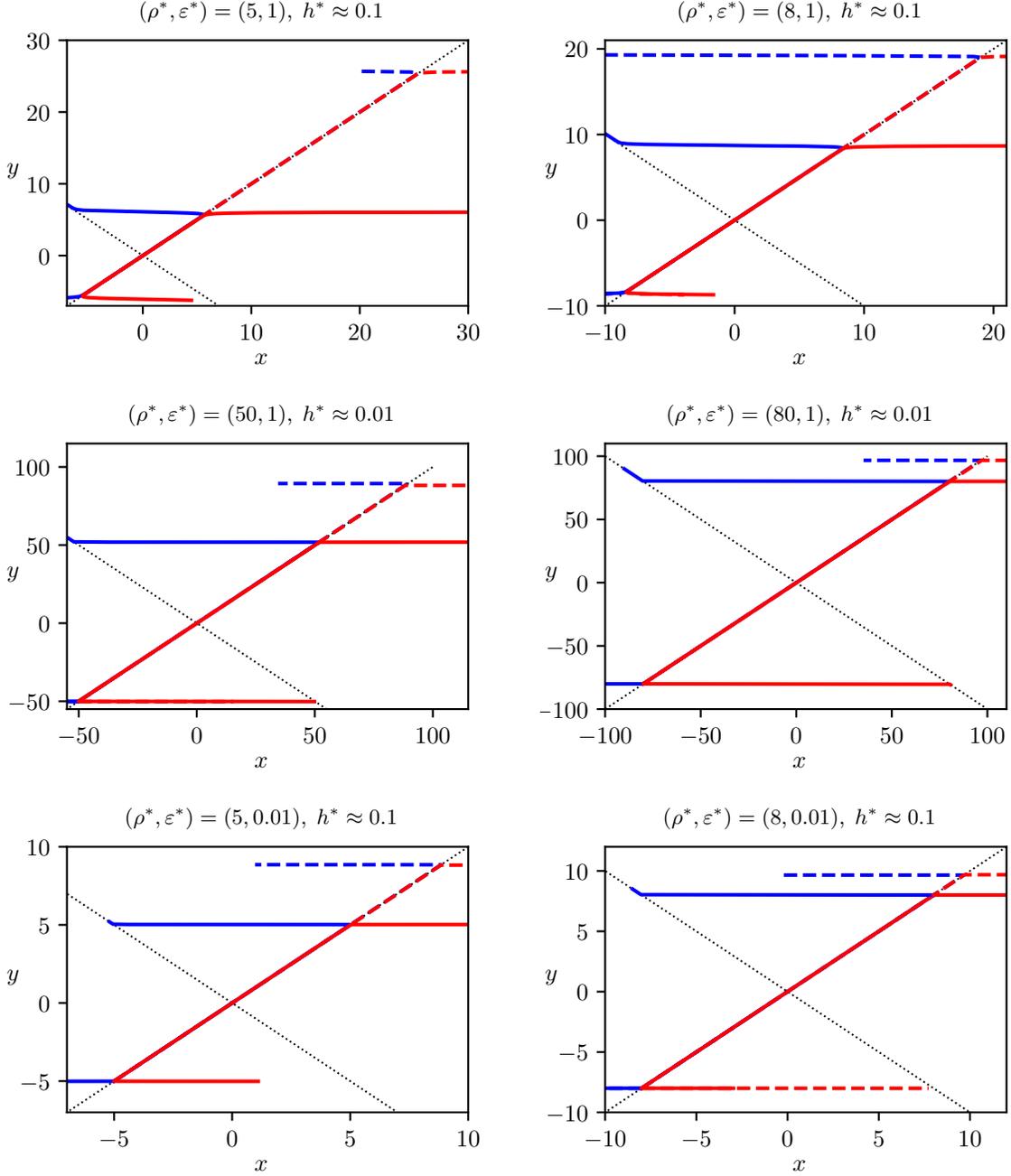
\begin{figure}[htbp]
    \centering

    \begin{tikzpicture}[scale=.99]
    \node (g1) at (-4,0){
    \includegraphics[scale=1]{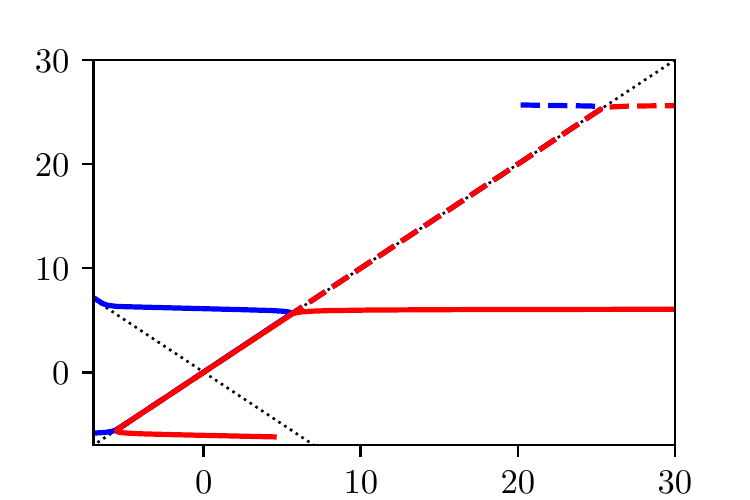}
    };
    \node[above= -6mm of g1] {
    \small
    $(\rho^*,\ve^*)=(5,1), \; h^*\approx 0.1 $
    };
    \node[left=-5mm of g1]{$y$};
    \node[below=-1mm of g1]{$x$};
    \node (g2) at (4,0){
    \includegraphics[scale=1]{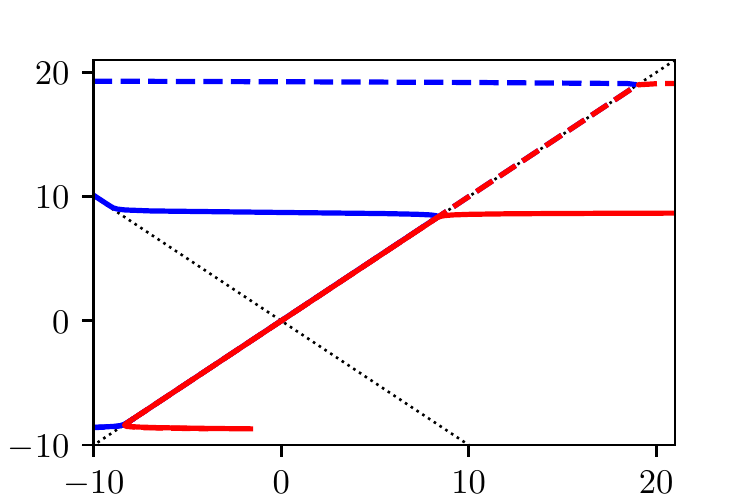}
    };
    \node[above= -6mm of g2] {
    \small
    $(\rho^*,\ve^*)=(8,1),\; h^*\approx 0.1$
    };
    \node[left=-5mm of g2]{$y$};
    \node[below=-1mm of g2]{$x$};
    
    \node (g3) at (-4,-6){
    \includegraphics[scale=1]{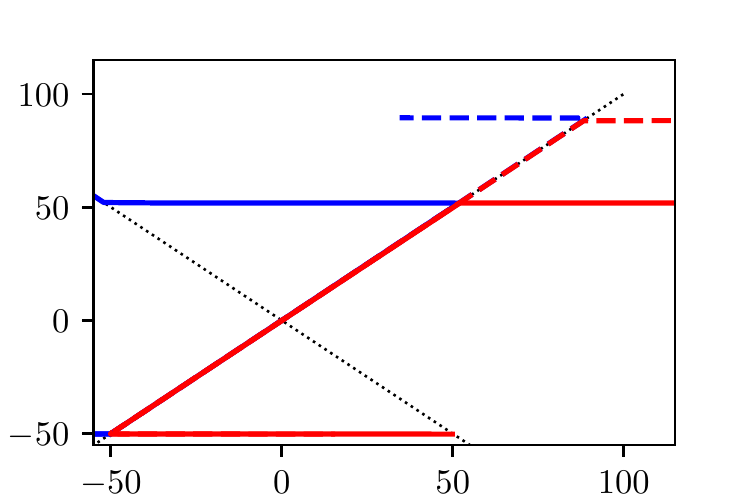}
    };
    \node[above= -6mm of g3] {
    \small
    $(\rho^*,\ve^*)=(50,1), \; h^*\approx 0.01 $
    };
    \node[left=-5mm of g3]{$y$};
    \node[below=-1mm of g3]{$x$};
    \node (g4) at (4,-6){
    \includegraphics[scale=1]{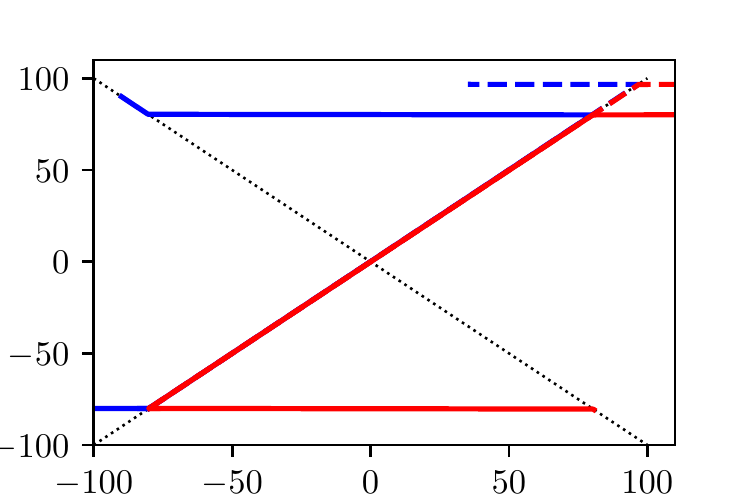}
    };
    \node[above= -6mm of g4] {
    \small
    $(\rho^*,\ve^*)=(80,1), \; h^*\approx 0.01 $
    };
    \node[left=-5mm of g4]{$y$};
    \node[below=-1mm of g4]{$x$};
    \node (g5) at (-4,-12){
    \includegraphics[scale=1]{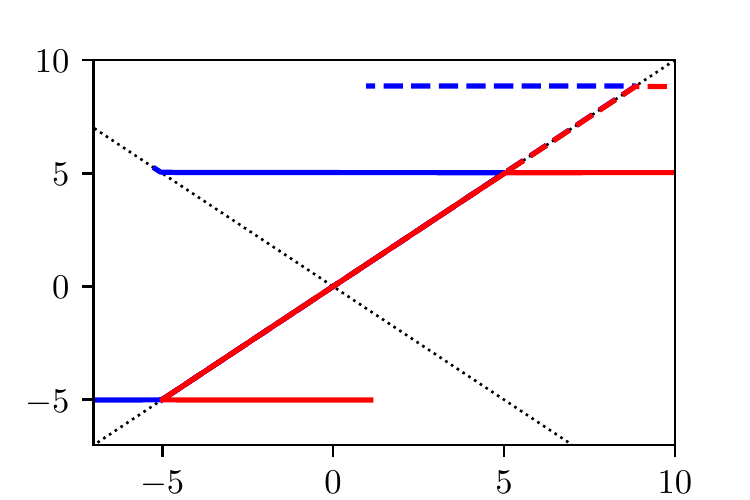}
    };
    \node[above= -6mm of g5] {
    \small
    $(\rho^*,\ve^*)=(5,0.01), \; h^*\approx 0.1 $
    };
    \node[left=-5mm of g5]{$y$};
    \node[below=-1mm of g5]{$x$};
    \node (g6) at (4,-12){
    \includegraphics[scale=1]{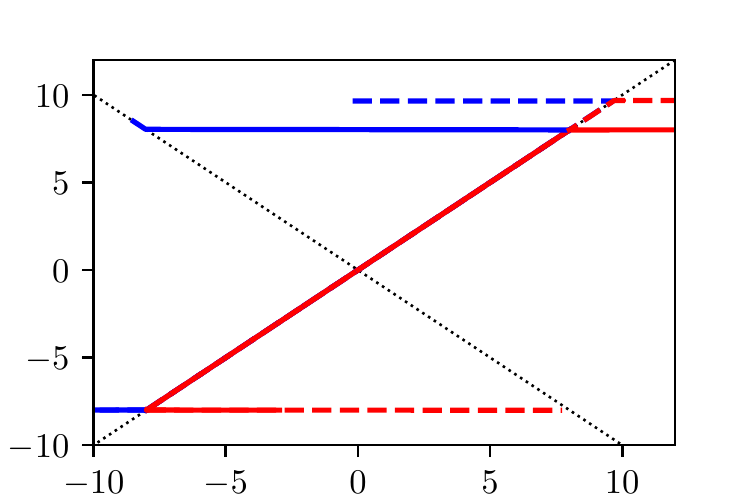}
    };
    \node[above= -6mm of g6] {
    \small
    $(\rho^*,\ve^*)=(8,0.01), \; h^*\approx 0.1 $
    };
    \node[left=-5mm of g6]{$y$};
    \node[below=-1mm of g6]{$x$};
    \end{tikzpicture}
    \caption{\blue{All plots show orbits that pass through the points $(x,y)=(-\rho^*,-\rho^*+10^{-4})$ (blue) and $(x,y)=(-\rho^*,-\rho^*-10^{-4})$ (red). The approximated value of $h^*$ is given by Figure \ref{fig:critical-triplet}. The first column corresponds to simulations with the Euler method, while the second column corresponds with Kutta's 3rd order method. Solid lines are produced with a sufficiently small discretization step ($h=10^{-3}$) to produce expected, close to symmetric, delay. Dashed lines correspond to a time step close (up to the 100th decimal digit) to the critical value $h^*$, see more details in Table \ref{tab:values}.}
    }
    \label{fig:delays}
\end{figure}

\setlength{\tabcolsep}{20pt}

\begin{table}[htbp]
    \centering
\begin{tabular}{lc}
    \multicolumn{2}{c}{Euler}\\
    \toprule
        $(\rho^*,\ve^*)$ & $h^*$  \\
        \toprule
        $(5,1)$  &  $0.104\cdots0<h^*<0.104\cdots7$\\
        \midrule
        $(50,1)$ & $0.009\cdots0<h^*<0.010\cdots1$\\
        \midrule 
        $(5,0.01)$ & $0.099\cdots0<h^*<0.100\cdots1$\\
    \end{tabular}

\vspace*{1cm}

     \begin{tabular}{lc}
    \multicolumn{2}{c}{Kutta 3rd order}\\
    \toprule
        $(\rho^*,\ve^*)$ & $h^*$  \\
        \toprule
        $(8,1)$  &  $0.100\cdots0<h^*<0.100\cdots4$\\
        \midrule
        $(80,1)$ & $0.010\cdots0<h^*<0.010\cdots4$\\
        \midrule 
        $(8,0.01)$ & $0.100\cdots3<h^*<0.100\cdots9$\\
    \end{tabular}
    \caption{\blue{Numerical values of the critical triplets used for the simulations of Figure \ref{fig:delays}. We remark that the surfaces of Figure \ref{fig:critical-triplet} are computed from the variational problem, and thus give only a good enough approximation of the true values of the critical triplet for the nonlinear system. The numerical values presented in this table are computed from the nonlinear system via a bisection method: we fix $(\rho^*,\ve^*)$ as in the table. Next, near the critical value $h^*$, if $h<h^*$ then the trajectory jumps in the correct direction while if $h>h^*$, the trajectory jumps to the opposite wrong direction. We have computed the critical discretization step size up to 100 significant decimals. Therefore, the dots $\cdots$ represent $97$ digits. For the simulations of Figure \ref{fig:delays} we have used the values $(\rho,\ve)$ as indicated in the first column of the tables and, for the discretization step $h$, the lower bound of $h^*$. }}
    \label{tab:values}
\end{table}

\clearpage

In Figure \ref{fig:delays} one can see that a step size close to the critical value indeed induces an extra delay on the onset of instability. Numerically, the least cumbersome situation is given for $(h,\ve)=(0.1,1)$. Hence, for such values one easily observes a longer delay for a step size close to the critical value (up to the $100$th digit). However, we also demonstrate in Figure \ref{fig:delays} the more general results from Section \ref{sec:rk_transcrit}; indeed, the delay becomes larger as $h\to h^*$ also for smaller combinations of $(h,\ve)$.

To finalize, we can compare the above simulations, using explicit RK-methods, with simulations obtained by using the Kahan method, as depicted in Figure \ref{fig:kahan}: the onset of instability is completely symmetric for different choices of $(h,\ve)$.

\begin{figure}[htbp]\centering
\begin{tikzpicture}
\node (g1) at (-4,3) {\includegraphics{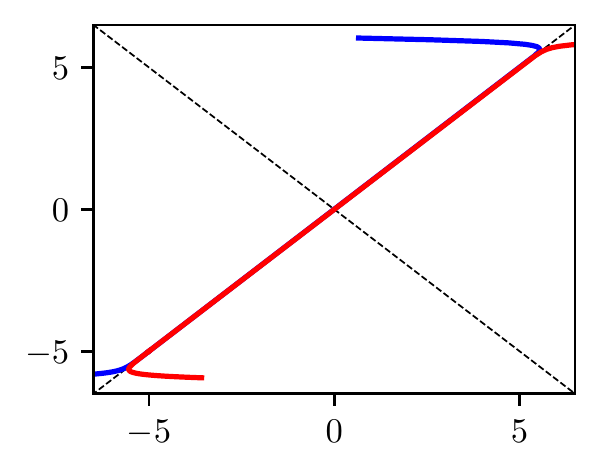}};
\node (g2) at ( 4,3) {\includegraphics{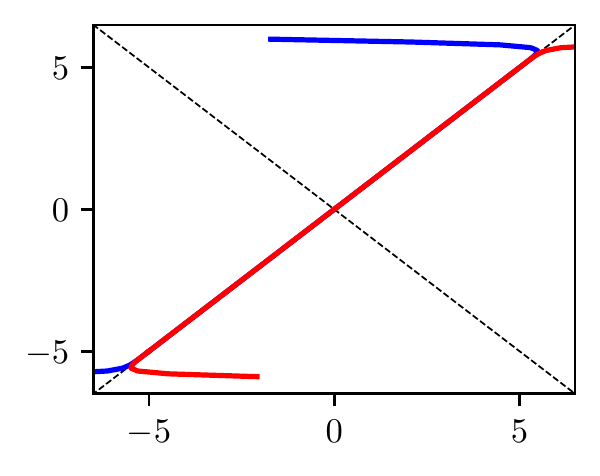}};
\node (g3) at (-4,-3) {\includegraphics{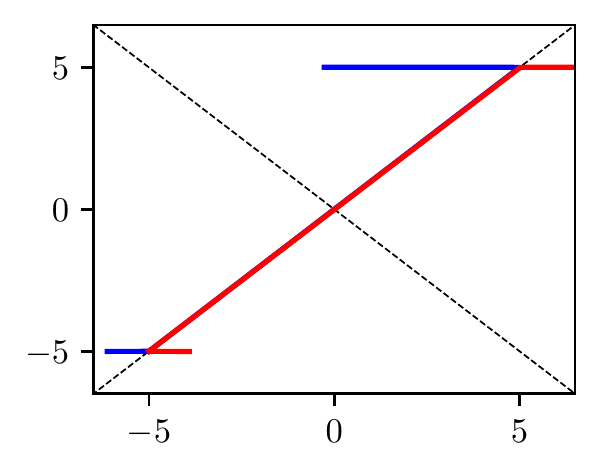}};
\node (g4) at ( 4,-3) {\includegraphics{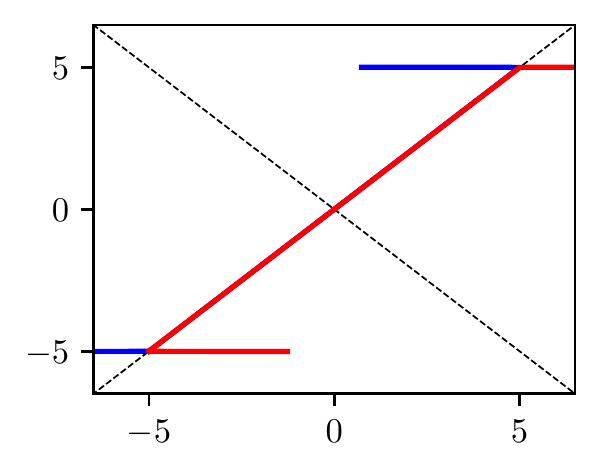}};

\node[left=-5mm of g1] {$y$};
\node[left=-5mm of g2] {$y$};
\node[left=-5mm of g3] {$y$};
\node[left=-5mm of g4] {$y$};

\node[below=-4mm of g1] {$x$};
\node[below=-4mm of g2] {$x$};
\node[below=-4mm of g3] {$x$};
\node[below=-4mm of g4] {$x$};

\node[above=-4.5mm of g1] {$(h,\ve)=(0.01,1)$};
\node[above=-4.5mm of g2] {$(h,\ve)=(0.1,1)$};
\node[above=-4.5mm of g3] {$(h,\ve)=(0.01,.01)$};
\node[above=-4.5mm of g4] {$(h,\ve)=(0.1,.01)$};

\end{tikzpicture}
\caption{\blue{Simulations of a fast-slow system with transcritical singularity, using the Kahan method, for several combinations of $(h,\ve)$ values. The blue curves correspond to the orbit passing through $(x,y)=(-5,-5+10^{4})$, while the red curves correspond to the orbit passing through $(x,y)=(-5,-5-10^{4})$. We observe that, as proved in section \ref{sec:Kahan}, the delay is symmetric.}}
\label{fig:kahan}
\end{figure}
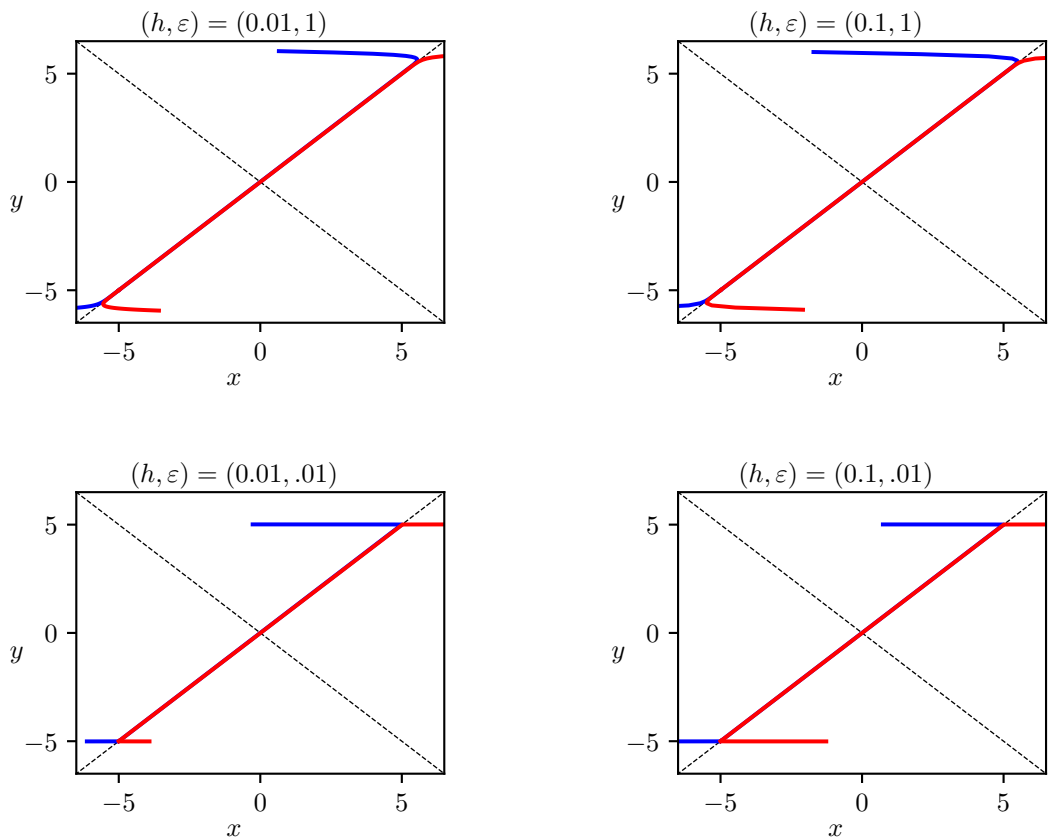

}

\section{Conclusion}\label{sec:discussion}
We have shown that explicit Runge-Kutta schemes can fail to provide accurate approximations of dynamical behaviour along canard trajectories for certain combinations of step size $h$ and entry coordinates $(x_0, y_0)$ in the plane. In fact, RK-schemes are prone to show much longer delays on the onset of bifurcation than one expects from the continuous-time cases. It is worth noting that for maps, such an extra delay is not a phenomenon due to time scale separation, but truly a discretization artefact. We have quantified this phenomenon for simple canonical forms of planar fast-slow problems: hence, we emphasize that, also for more complicated systems, one should be cautious when simulating trajectories close to canards and trying to find entry-exit relations.

Additionally, we have proven that an implicit Runge-Kutta scheme like the Kahan method, which gives an explicit birational map for quadratic vector fields, preserves the symmetry of the linearization around planar canards. 
\blue{At the hand of the pitchfork example, we have studied a more general class of A-stable, symmetric second order methods, to which the Kahan method belongs, showing that, depending on an additional parameter, some of these methods preserve and others reverse the continuous-time canards. In particular, we have demonstrated that the discretization of the nonlinear dynamics and its description via linearization along special trajectories is sensitive to dynamical intricacies which are not yet fully understood and cannot be easily resolved around A-stability.
Structure-preserving discretization for non-Hamiltonian systems is still investigated step by step (see also \cite{Engeletal2019}) and a deeper understanding is still developed.} 

\blue{The extension of the ideas developed in this paper to higher dimensional systems requires careful consideration. On the one hand, if the dynamics of a high dimensional system can be reduced to a planar one (e.g.~via center manifold reduction), where a singularity as those considered here appears, then we can expect to observe the phenomena presented in this article. On the other hand, when considering higher dimensional systems many additional singular bifurcations have to be analyzed. As a direct extension to our work, one may want to consider discretizations of a fast-slow system with a singular Hopf bifurcation where the delayed onset of instability has been studied by extending the continuous time variable to the complex plane \cite{hayes2016geometric,neishtadt1987persistence,neishtadt1988persistence}. It would be particularly interesting to compare such an analysis with the study of a fast-slow version of the Neimark-Sacker bifurcation. In a similar context, we believe that our work also motivates a more general analysis of $2$-parameter families of dynamic bifurcations for maps.}


\section*{Acknowledgements}
The authors gratefully acknowledge C. Kuehn for pointing out the finite form of the Jensen inequality. ME is supported by the grant SFB/TRR 109 ``Discretization in Geometry and Dynamics" sponsored by the German Research Foundation (DFG). HJK is supported by the Alexander von Humboldt Foundation.

\bibliographystyle{abbrv}
\bibliography{bib_summary}

\end{document}